\documentclass[letter,12pt]{article}

\usepackage[english]{babel}
\usepackage{amssymb}
\usepackage{amsthm}
\usepackage{amsfonts}
\usepackage{amsmath}
\usepackage{anysize}
\usepackage{bbm}
\usepackage{abstract}
\marginsize{2.5cm}{2.5cm}{0.9cm}{1.9cm}
\usepackage{xcolor}

\newtheorem{theorem}{Theorem}[section]
\newtheorem{lemma}[theorem]{Lemma}
\newtheorem{prop}[theorem]{Proposition}
\newtheorem{coro}[theorem]{Corollary}

\theoremstyle{definition}

\newtheorem{remark}[theorem]{Remark}

\newcommand{\NN}{\mathbb{N}}

\newcommand{\Ec}{\mathcal{E}}

\newcommand{\setdef}{\ \vert \ }
\newcommand{\vep}{\varepsilon}
\newcommand{\psh}{{\rm PSH}}

\newcommand{\capa}{{\rm Cap}}
\newcommand{\capi}{{\rm Cap}_{\phi}}
\newcommand{\capo}{{\rm Cap}_{\omega}}

\newcommand{\vol}{{\rm Vol}}

\newcommand{\ddbar}{\partial\bar\partial}

\newcommand{\AM}{{\rm I}}
\newcommand{\AMO}{{\rm I}_{\phi}}

\newcommand{\Amp}{{\rm Amp}}
\newcommand{\id}{\mathbbm{1}}

\usepackage{hyperref}
\hypersetup{
    bookmarks=true,         
    unicode=false,          
    pdftoolbar=true,        
    pdfmenubar=true,        
    pdffitwindow=false,     
    pdfstartview={FitH},    
    pdftitle={Monotonicity of  non-pluripolar products and complex Monge-Amp\`ere equations with prescribed singularity},    
    pdfauthor={T. Darvas, E. Di Nezza, C.H. Lu},     
    colorlinks=true,       
   linkcolor=black,          
    citecolor=black,        
    filecolor=black,      
    urlcolor=black}           

\title{Monotonicity of  non-pluripolar products and  complex Monge-Amp\`ere equations with prescribed singularity }

\author{Tam\'as Darvas, Eleonora Di Nezza, Chinh H. Lu}
\date{\vspace{-0.2in}}

\begin{document}
\maketitle
  
\begin{abstract}
We establish the monotonicity property for the mass of non-pluripolar products on compact K\"ahler manifolds, and we initiate the study of complex Monge-Amp\`ere type equations with prescribed singularity type. Using the variational method  of Berman-Boucksom-Guedj-Zeriahi we prove existence and uniqueness of solutions with small unbounded locus. We give applications to K\"ahler--Einstein metrics with prescribed singularity, and we show that the log-concavity property holds for non-pluripolar products with small unbounded locus.
\end{abstract}


\section{Introduction and main results}

Let $X$ be a compact K\"ahler manifold of complex dimension $n$, and let $\theta$ be a smooth closed real $(1,1)$-form on $X$ such that $\{ \theta\}$ is big. Broadly speaking, the purpose of this article is threefold. First, we develop the potential theory of non-pluripolar products without any restrictions on the singularity type, by combining techniques of Witt Nystr\"om \cite{WN17} and previous work of the authors \cite{DDL16}. Second, given $\phi \in \textup{PSH}(X,\theta)$, we introduce and study the spaces $\mathcal E(X,\theta,\phi)$ and $\mathcal E^1(X,\theta,\phi)$, generalizing the content of \cite{BEGZ10} to the relative framework. These latter spaces contain potentials that are slightly more singular than $\phi$, and satisfy a (relative) full mass/finite energy condition. Lastly, with sufficient potential theory developed, we focus on the variational study of the complex Monge-Amp\`ere equation
\begin{equation}\label{eq: CMAE_intr}
(\theta + i\ddbar u)^n = f \omega^n,
\end{equation}
where $f \geq 0$, $f \in L^p(\omega^n), \ p > 1,$  and the singularity type of $u \in \textup{PSH}(X,\theta)$ is the same as that of $\phi$.
As it will turn out, this equation is well posed only for potentials $\phi$ with a certain type of ``model'' singularity, that includes the case of analytic singularities, and we provide existence of unique solutions with small unbounded locus. As we will see, on the right hand side of \eqref{eq: CMAE_intr} one may even consider more general (non-pluripolar) Radon measures. 

When $\theta$ is a K\"ahler form, $f>0$ is smooth, and $\phi=0$, the above equation was solved (with smooth solutions) by Yau \cite{Au78,Ya78}, resolving the famous Calabi conjecture. Using both a priori estimates and pluripotential theory, this result was later extended in many different directions (see \cite{Kol98,Kol03,GZ07,BEGZ10,BBGZ13,Ber13,PS14}). Our approach seems to unify all existing works (in the compact setting), under the theme of solutions with arbitrary prescribed (model) singularity type.  

At the end of the paper, we give applications of our results to singular K\"ahler--Einstein metrics and establish the log-concavity property for certain non-pluripolar products. Other applications will be treated in a sequel.

Though we will work in the general framework of big cohomology classes throughout the paper, we note that all our results seem to be new in the particular case of K\"ahler classes as well. 

\paragraph{Monotonicity of non-pluripolar products and relative finite energy.} Unless otherwise specified, we fix a background K\"ahler structure $(X,\omega)$ for the remainder of the paper.

We say that a potential $u\in L^1(X,\omega^n)$  is $\theta$-plurisubharmonic ($\theta$-psh) if locally  $u$ is the difference of a plurisubharmonic  and a smooth function, and $\theta_u :=\theta+i\ddbar u \geq 0$ in the sense of currents. The set of $\theta$-psh potentials is denoted by $\textup{PSH}(X,\theta)$. We say that $\{\theta \}$ is \emph{pseudoeffective} if $\textup{PSH}(X,\theta)$ is non-empty. Along these lines, $\{\theta\}$ is \emph{big} if $\textup{PSH}(X,\theta-\varepsilon \omega)$ is non-empty for some $\varepsilon >0$. 

If $u$ and $v$ are two  $\theta$-psh functions on $X$, then $u$ is said to be \emph{less singular} than $v$ if $v\leq u+C$ for some $C\in \Bbb R$. We say that $u$ has the same singularities as $v$, if $u$ is less singular than $v$, and $v$ is less singular than $u$. This defines an equivalence relation on $\textup{PSH}(X,\theta)$, whose equivalence classes are the \emph{singularity types}
$[u], \ u \in \textup{PSH}(X,\theta)$.


Given closed positive $(1,1)$-currents $T_1:=\theta^1_{u_1},...,T_p:=\theta^p_{u_p}$, where $\theta^j$ are closed smooth real $(1,1)$-forms, generalizing the construction of Bedford-Taylor \cite{BT87} in the local setting, it has been shown in \cite{BEGZ10} that one can define the \emph{non-pluripolar product}  of these currents:
\[
\theta^1_{u_1}\wedge \ldots \wedge \theta^p_{u_p}:= \langle T_1\wedge...\wedge T_p\rangle.
\]
The resulting positive $(p,p)$-current does not charge pluripolar sets and it is \emph{closed}. For a $\theta$-psh function $u$, the \emph{non-pluripolar complex Monge-Amp{\`e}re measure} of $u$ is simply $\theta_u^n:=\theta_u\wedge \ldots\wedge \theta_u.$

It has recently been proved by Witt Nystr\"om that the complex Monge-Amp\` ere mass of $\theta$-psh potentials decreases as the singularity type increases  \cite[Theorem 1.2]{WN17}.
Our  main result about monotonicity of non-pluripolar products generalizes this result to the case of different cohomology classes $\{\theta^j\}$, fully proving  what was conjectured by Boucksom-Eyssidieux-Guedj-Zeriahi (see the comments after   \cite[Theorem 1.16]{BEGZ10} in which they prove that the result holds for potentials with small unbounded locus):

\begin{theorem}\label{thm1}
	Let $\theta^j, j \in \{1,\ldots,n\}$ be smooth closed real $(1,1)$-forms on $X$. Let $u_j,v_j \in \textup{PSH}(X,\theta^j)$ such that $u_j$ is less singular than $v_j$ for all $j \in \{1,\ldots,n\}$. Then 
	\[
	\int_X \theta^1_{u_1} \wedge \ldots \wedge \theta^n_{u_n}  \geq  \int_X \theta^1_{v_1} \wedge \ldots \wedge \theta^n_{v_n}. 
	\]
\end{theorem}

To prove the above theorem, we first need to generalize the main convergence theorems of Bedford-Taylor theory (\cite{BT87}, see also  \cite{X96,X09}). This is done collectively in the next result, further elaborated in Theorem \ref{thm: lsc of MA measures} below:

\begin{theorem}\label{thm2}
Let $\theta^j, j \in \{1,\ldots,n\}$ be smooth closed real $(1,1)$-forms on $X$. Suppose that  we have $u_j,u_j^k\in \textup{PSH}(X,\theta^j)$ such that  $u^k_j \to u_j$ in capacity as $k \to \infty$, and 
\begin{equation}\label{eq: main_global_mass_semi_cont}
\int_ X \theta^1_{u_1} \wedge \ldots \wedge \theta^n_{u_n} \geq \limsup_{k\rightarrow \infty} \int_X \theta^1_{u^k_1} \wedge \ldots \wedge \theta^n_{u^k_n}.
\end{equation}
Then $\theta^1_{u^k_1} \wedge \ldots \wedge \theta^n_{u^k_n}  \to  \theta^1_{u_1} \wedge \ldots \wedge \theta^n_{u_n}$ in the weak sense of measures. 
\end{theorem}
\noindent We recall that a sequence $\{u_k\}_k$ converges in capacity to $u$ if for any $\delta>0$ we have 
$$\lim_{k\rightarrow +\infty} \capo \{|u_k-u|\geq \delta\}=0, $$
where $\capo$ is the Monge-Amp\`ere capacity associated to $\omega$ (see \cite[Definition 4.23]{GZ17}).

We note that condition \eqref{eq: main_global_mass_semi_cont} is  necessary in this generality, even in the K\"ahler case. Indeed, if $u \in \textup{PSH}(X,\omega)$ is a pluricomplex Green potential, then the cut-offs $u_j := \max(u,-j) \in \textup{PSH}(X,\omega)$ satisfy $u_j \searrow u$. However $\int_X \omega_{u_j}^n = \int_X \omega^n >0$ for all $j$, and $\int_X \omega_u^n=0$, hence $\omega^n_{u_j}$ cannot converge to $\omega_u^n$ weakly. 

As noted above, Theorem \ref{thm2} generalizes classical theorems of Bedford-Taylor (when $u_j^k,u_j$ are uniformly bounded) and also results from \cite{BEGZ10} (when $u_j^k,u_j$ have full mass). In both of these cases, there are severe restrictions on the singularity class of the potentials  $u_j^k,u_j$. On the other hand, the above theorem shows that there is no need for restrictions on singularity type of the potentials involved. Instead, one needs only a semicontinuity condition on the total masses.\medskip

To develop the variational approach to equation \eqref{eq: CMAE_intr}, with the above general results in hand, we initiate the study of relative full mass/relative finite energy currents. Let $\phi \in \textup{PSH}(X,\theta)$. We say that $v \in \textup{PSH}(X,\theta)$ has \emph{full mass relative} to $\phi$ ($v \in \mathcal E(X,\theta,\phi)$) if $v$ is more singular than $\phi$ and $\int_X \theta_v^n = \int_X \theta_\phi^n$. In our investigation of these classes, the following well known envelope constructions will be of great help:
$$\textup{PSH}(X,\theta) \ni \psi \to P_\theta(\psi,\phi), \ P_\theta[\psi](\phi),  \ P_\theta[\psi] \in \textup{PSH}(X,\theta).$$ 
These were introduced by Ross and Witt Nystr\"om \cite{RWN} in their construction of geodesic rays, building on ideas of  Rashkovskii and Sigurdsson \cite{RS} in the local setting. Due to the frequency of these operators appearing in this work, we choose to follow slightly different notations. The starting point is the  ``rooftop envelope'' $P_\theta(\psi,\phi):=\sup\{v \in \textup{PSH}(X,\theta), \ v \leq \min(\psi,\phi) \}$. This allows to introduce
$$P_\theta[\psi](\phi) := \Big(\lim_{C \to +\infty}P_\theta(\psi+C,\phi)\Big)^*,$$
and is easy to see that $P_\theta[\psi](\phi)$ only depends on the singularity type of $\psi$. When $\phi=0$ or $\phi = V_\theta$, we will simply write $P_\theta[\psi]:=P_\theta[\psi](0)=P_\theta[\psi](V_\theta)$, and we refer to this potential as the \emph{envelope of the singularity type} $[\psi]$.

Using the techniques of our recent work \cite{DDL16}, we can give a generalization of \cite[Theorem 3]{Dar13} (paralleling \cite[Theorem 1.2]{DDL16}). This result characterizes membership in $\mathcal E(X,\theta,\phi)$ solely in terms of singularity type: 
 \begin{theorem}\label{thm3}Suppose $\phi \in \textup{PSH}(X,\theta)$ and $\int_X \theta_\phi^n >0$. The following are equivalent:\\
\noindent (i) $u \in \mathcal E(X,\theta,\phi)$.\\
\noindent (ii) $\phi$ is less singular than $u$, and $P_\theta[u](\phi)=\phi$.\\
\noindent (iii) $\phi$ is less singular than $u$, and $P_\theta[u]=P_\theta[\phi]$.
\end{theorem}

Without the \emph{non-zero mass} condition $\int_X \theta_{\phi}^n>0$ this characterization cannot hold (see Remark \ref{ex: zero mass}). 
The equivalence between (i) and (iii) in the above theorem shows that $P_\theta[u]$ is the same potential for any $u \in \mathcal E(X,\theta,\phi)$, and equals to $P_\theta[\phi]$. Given this and the inclusion  $\mathcal E(X,\theta,\phi) \subset \mathcal E(X,\theta,P_\theta[\phi])$, one is  tempted to consider only potentials $\phi$ in the image of the operator $\psi \to P_\theta[\psi]$, when studying the classes of relative full mass $\mathcal E(X,\theta,\phi)$. 
These potentials seemingly play the same role as $V_\theta$, the potential with minimal singularities from \cite{BEGZ10}. Implementation of this idea will be further motivated by the results of the next paragraph.

In addition to the above result, we also establish analogs of many classical results for $\mathcal E(X,\theta,\phi)$, like the comparison, maximum and domination principles. Some of these are routine while others, like the domination principle, require new techniques and a more involved analysis compared to the existing literature (see Proposition \ref{prop: domination principle}). 

\paragraph{Complex Monge-Amp\`ere equations with prescribed singularity.}  With the potential theoretic tools developed, we focus on solving \eqref{eq: CMAE_intr}. A simple minded example shows that this equation is not well posed for arbitrary $\phi \in \textup{PSH}(X,\theta)$ (see the introduction of Section \ref{sec 4}). Instead, one needs to consider only potentials $\phi$ that are fixed points of the operator $\psi \to P_\theta[\psi]$, i.e. $\psi = P_\theta[\psi]$.  Such potentials $\psi$ will be called \emph{model potentials}, and their singularity types $[\psi]$ will be called \emph{model type singularities}. In this direction we have the following result:

\begin{theorem}\label{thm4} Suppose $\phi \in \textup{PSH}(X,\theta)$  has small unbounded locus, and $\phi = P_\theta[\phi]$. Let  $f \in L^p(\omega^n), p > 1$ such that $f \geq 0$  and $\int_X f \omega^n=\int_X \theta_{\phi}^n > 0$. Then the following hold:\\
(i) There exists  $u\in \textup{PSH}(X,\theta)$, unique up to a constant, such that $[u]=[\phi]$ and 
\begin{equation}\label{eq: main_thm_eq1}
\theta_u^n=f \omega^n.
\end{equation}
(ii) For any $\lambda >0$  there exists a unique $v\in \textup{PSH}(X,\theta)$,  such that $[v]=[\phi]$ and
\begin{equation}\label{eq: main_thm_eq2}
\theta_v^n=e^{\lambda v}f \omega^n.
\end{equation}
\end{theorem}

That $\phi$ has \emph{small unbounded locus} means that $\phi$ is locally bounded outside a closed complete  pluripolar set $A \subset X$. It will be interesting to see if this condition is simply technical, or otherwise necessary. This seemingly extra condition on $\phi$ does have some benefits. Indeed, since in this setting solutions are locally bounded on $X \setminus A$, one can  interpret \eqref{eq: main_thm_eq1} and \eqref{eq: main_thm_eq2} in the following simply way: $u$ and $v$ satisfy the equations \eqref{eq: main_thm_eq1} and \eqref{eq: main_thm_eq2} on $X \setminus A$, in the sense of Bedford-Taylor.

\begin{remark} As argued in Theorem \ref{thm: naturality_of_model}, if \eqref{eq: main_thm_eq1} can be solved for all $f \in L^{p}(X), \ p >1$ (with the constraint $[u]=[\phi]$)  then $\phi$ \emph{must have} model type singularity. Consequently, our choice of $\phi$ in the above theorem is not ad hoc, but truly natural!
\end{remark}

In our study of the above equations we will start with a much more general context. In particular, we will show in Theorem \ref{thm: existence lambda=0} and Theorem \ref{thm: existence_MA_eq_exp} below that instead of $f\omega^n$ one can consider on the right hand side of \eqref{eq: main_thm_eq1} and \eqref{eq: main_thm_eq2} non-pluripolar measures, thereby generalizing \cite[Theorem A,Theorem D]{BEGZ10}.

\begin{remark}\label{rem: examples} Naturally, $V_\theta = P_\theta[V_\theta]$, but our reader may wonder if there are other interesting enough potentials with model type singularity. We believe this to be the case, as evidenced below :\vspace{0.1cm}

$\bullet$ By Theorem \ref{thm: ceiling coincide envelope non collapsing} below, $P_\theta[\psi]=P_\theta[P_\theta[\psi]]$ for any $\psi \in \textup{PSH}(X,\theta)$ with $\int_X \theta_\psi^n >0$. In particular, $P_\theta[\psi]$ is a model potential, giving an abundance of potentials with model type singularity.\vspace{0.1cm}

$\bullet$ By Proposition \ref{prop: Lp example} below, if $\psi \in \textup{PSH}(X,\theta)$ has small unbounded locus, and $\theta_\psi^n/ \omega^n  \in L^p(\omega^n), \ p > 1$ with $\int_X  \theta_\psi^n >0$, then $\psi$ has model type singularity.\vspace{0.1cm}

$\bullet$ All  \emph{analytic singularity types} (those that can be locally written as $c \log \big(\sum_j |f_j|^2\big) + g$, where $f_j$ are holomorphic, $c > 0$ and $g$ is smooth) are of model type (\cite[Remark 4.6]{RWN}, \cite{RS}, see also Proposition \ref{prop: analytic example}). In particular,  discrete logarithmic singularity types are of model type, making connection with pluricomplex \vspace{0.1cm} Green currents \cite{CG09,PS14,RS}. 

$\bullet$ By \cite{RWN,Dar13,DDL16}, potentials with model type singularity naturally arise as degenerations along geodesic rays and in particular along test configurations.
\end{remark}
Complex Monge-Amp\`ere equations with  bounded/minimally singular solutions have been intensely studied in the past (\cite{Kol98,Kol03,GZ07,BEGZ10,BBGZ13}, to name only a few works in a fast expanding literature). To our knowledge, in the compact case, only the paper \cite{PS14} discusses at length solutions that are not ``minimally singular'', without severe restrictions on the right hand side of the equation. They treat the case of solutions to \eqref{eq: main_thm_eq1} with isolated algebraic singularities in the K\"ahler case, with a view toward constructing pluricomplex Green currents on $X$. Given the specific setting, \cite[Theorem 3]{PS14} obtains more precise regularity estimates compared to ours, using blowup techniques. In our general framework better estimates are likely not possible. However for smooth $f$, we suspect that away from the singularity locus our solution $u$ should be as regular as $\phi$ (up to order two). For a general result on the regularity of certain model potentials we refer to \cite[Theorem 1.1]{RWN2}.  

Lastly, let us mention that  in \cite[Section 4]{Ber13} solutions to complex Monge-Amp\`ere equations with divisorial singularity type are used in the construction/approximation of geodesic rays corresponding to certain test configurations. In \cite[Section 5]{Ber13} Berman speculates that solutions with more general singularity type should allow for better understanding of degenerations along test configurations/geodesic rays, and we believe our treatise will lead to more results of this flavor.

In addition to the results in the compact setting mentioned above, finding singular/non-bounded solutions to the related Dirichlet problem on domains in $\mathbb{C}^n$, or more generally on compact manifolds  with boundary, was studied by a number of authors. We only mention \cite{L83, BD88, Gu98, PS09, PS10} to highlight a few works in a fast expanding literature.

\paragraph{Applications.} Solutions of complex Monge--Amp\`ere equations are linked to existence of special K\"ahler metrics. In particular, we can think of the solution to \eqref{eq: main_thm_eq1} as a potential with prescribed singularity type and prescribed Ricci curvature in the philosophy of the Calabi-Yau theorem.
As an immediate application of our solution to \eqref{eq: main_thm_eq2} we obtain existence of  singular \emph{K\"ahler-Einstein} (KE) metrics with prescribed singularity type on K\"ahler manifolds of general type. An analogous result also holds on Calabi-Yau manifolds as well, via solutions of \eqref{eq: main_thm_eq1}.

\begin{coro} Let $X$ be a smooth projective variety of general type ($K_X >0$) and let $h$ be a smooth Hermitian metric on $K_X$ with $\theta := \Theta(h)>0$. Suppose also that $\phi \in \textup{PSH}(X,\theta)$ is a model potential, has small unbounded locus and $\int_X \theta_\phi^n>0$. Then there exists a unique singular KE metric $he^{-\phi_{KE}}$ on $K_X$ ($\theta_{\phi_{KE}}^n = e^{\phi_{KE}+f_\theta} \theta^n$, where $f_\theta$ is the Ricci potential of $\theta$ satisfying  $\textup{Ric} \ \theta=\theta+i\ddbar f_\theta$), with $\phi_{KE} \in \textup{PSH}(X,\theta)$ having the same singularity type as $\phi$.
\end{coro}

As another application we confirm the log-concavity conjecture \cite[Conjecture 1.23]{BEGZ10} in the case of currents with potentials having small unbounded locus:

\begin{theorem}\label{thm5}
	Let $T_1,...,T_n$ be positive closed $(1,1)$-currents  on a compact K\"ahler manifold $X$.  Assume that each $T_j$ has a potential with small unbounded locus. Then 
	\[
	\int_X \langle T_1 \wedge \ldots \wedge T_n\rangle \geq \left(\int_X \langle T_1^n\rangle \right)^{\frac{1}{n}} \ldots \left(\int_X \langle T_n^n\rangle \right)^{\frac{1}{n}}.
	\] 
\end{theorem}

\paragraph{Possible future directions.}
It is well known that for $\lambda <0$ the equation \eqref{eq: main_thm_eq2} does not always have a solution. More importantly, solvability of this equation is tied together with existence of KE metrics on Fano manifolds. It would be interesting to see if the techniques of \cite{DR17} apply to give characterizations for existence of KE metrics with prescribed singularity type in terms of energy properness.

By \cite{Dar13,DDL16} the geometry of geodesic rays and properties of (relative) full mass potentials seems to be intimately related. In a future work we will explore this avenue further, by introducing a metric geometry on the space of singularity types, via the constructions of \cite{Dar13,DDL16}. By understanding the metric properties of this space, we hope to study degenerations of singularity types along complex Monge-Amp\`ere equations. 

\paragraph{Organization of the paper.} Most of our notation and terminology carries over from \cite{DDL16}, and we refer the reader to the introductory sections of this work. In Section \ref{sec 2} we prove Theorem \ref{thm1} and Theorem \ref{thm2}. In Section \ref{sec 3} we develop the theory of the relative full mass classes $\Ec(X, \theta, \phi)$ and we exploit properties of envelopes to prove Theorem \ref{thm3}. In Section \ref{sec 4} we generalize the variational methods of \cite{BBGZ13} to prove Theorem \ref{thm4}. Finally, Theorem \ref{thm5} is proved in Section \ref{sec 5}. 

\section{The monotonicity property and convergence of non-pluripolar products}\label{sec 2}

To begin, from the main result of \cite{WN17} we deduce the following proposition:
\begin{prop}
	\label{prop: comparison generalization}
	Let $\theta^j, j \in \{1,\ldots,n\}$ be smooth closed real $(1,1)$-forms on $X$ whose cohomology classes are pseudoeffective. Let $u_j,v_j \in \textup{PSH}(X,\theta^j)$ such that $u_j$ has the same singularity type as $v_j, \ j \in \{1,\ldots,n\}$. Then 
	\[
	\int_X \theta^1_{u_1} \wedge \ldots \wedge \theta^n_{u_n}  = \int_X \theta^1_{v_1} \wedge \ldots \wedge \theta^n_{v_n}. 
	\]
\end{prop}
The proof of this result uses the arguments in \cite[Corollary 2.15]{BEGZ10}.
\begin{proof} 
First we note that we can assume that the classes $\{\theta^j\}$ are in fact big. Indeed, if this is not the case we can just replace each $\theta^j$ with $\theta^j + \varepsilon \omega$, and using the multi-linearity of the non-pluripolar product (\cite[Proposition 1.4]{BEGZ10}) we can let $\varepsilon \to 0$ at the end of our argument to conclude the statement for pseudoeffective classes.

For each $t\in\Delta= \{t=(t_1,...,t_n) \in \mathbb{R}^n \setdef t_j > 0\}$ consider $u_t:=\sum_{j} t_ju_{j}$, $v_t:=\sum_{j} t_jv_{j}$ and $\theta^t:=\sum_{j} t_j\theta^j$.  Clearly,  $\{\theta^t\}$ is big, and $u_t$ has the same singularities as $v_t$. Hence it follows from \cite[Theorem 1.2]{WN17} that $\int_X (\theta^t_{u_t})^n =\int_X (\theta^t_{v_t})^n$ for all $t\in \Delta$. On the other hand, using multi-linearity of the non-pluripolar product again (\cite[Proposition 1.4]{BEGZ10}), we see that both $t \to \int_X (\theta^t_{u_t})^n$  and $t \to \int_X (\theta^t_{v_t})^n$ are homogeneous polynomials of degree $n$. Our last identity  forces all the coefficients of these polynomials to be equal, giving the statement of our result. 
\end{proof}

We recall a classical convergence theorem from Bedford-Taylor theory. We refer to \cite[Theorem 4.26]{GZ17} for a proof of this result, which is merely slight generalization of \cite[Theorem 1]{X96}.  

\begin{prop}\label{prop: xing_conv} Let $\Omega \subset \Bbb C^n$ be an open set. Suppose $\{f_j\}_j$ are uniformly bounded quasi-continuous
functions which converge in capacity to another quasi-continuous function $f$ on $\Omega$. Let $\{ u^j_1\}_j, \{ u^j_2\}_j,\ldots, \{ u^j_n\}_j$ be uniformly bounded plurisubharmonic functions on $\Omega$, converging in capacity to $u_1, u_2, \ldots, u_n$
respectively. Then we have the following weak convergence of measures:
$$f_j i\partial \bar \partial u_1^j \wedge i\partial \bar \partial u_2^j \wedge \ldots \wedge i\partial \bar \partial u_n^j \to f i\partial \bar \partial u_1 \wedge i\partial \bar \partial u_2 \wedge \ldots \wedge i\partial \bar \partial u_n.$$
\end{prop}

The following lower-semicontinuity property of non-pluripolar products will be key in the sequel: 
\begin{theorem}
	\label{thm: lsc of MA measures}
	Let $\theta^j, j \in \{1,\ldots,n\}$ be smooth closed real $(1,1)$-forms on $X$ whose cohomology classes are big. Suppose that for all $j \in \{1,\ldots,n\}$  we have $u_j,u_j^k\in \textup{PSH}(X,\theta^j)$ such that  $u^k_j \to u_j$ in capacity as $k \to \infty$. Then for all positive bounded quasi-continuous functions $\chi$ we have
	\[
	\liminf_{k\to +\infty} \int_X \chi \theta^1_{u^k_1} \wedge \ldots \wedge \theta^n_{u^k_n}  \geq  \int_X \chi  \theta^1_{u_1} \wedge \ldots \wedge \theta^n_{u_n}. 
	\]
If additionally,  
	\begin{equation}\label{eq: global_mass_semi_cont}
	\int_ X \theta^1_{u_1} \wedge \ldots \wedge \theta^n_{u_n} \geq \limsup_{k\rightarrow \infty} \int_X \theta^1_{u^k_1} \wedge \ldots \wedge \theta^n_{u^k_n},
	\end{equation}
then $\theta^1_{u^k_1} \wedge \ldots \wedge \theta^n_{u^k_n}  \to  \theta^1_{u_1} \wedge \ldots \wedge \theta^n_{u_n}$ in the weak sense of measures on $X$. 
\end{theorem}
\begin{proof}
	Set $\Omega:=\bigcap_{j=1}^n \Amp(\theta^j)$ and fix an open relatively compact subset $U$ of $\Omega$.  Then  the functions $V_{\theta^j}$ are bounded on $U$. We now use a classical idea in pluripotential theory. Fix $C>0,\vep>0$ and consider
	\[
	f_j^{k,C,\vep}:= \frac{\max(u_j^{k}-V_{\theta^j}+C,0)}{\max(u_j^{k}-V_{\theta^j}+C,0)+\vep}, \ j=1,...,n, \ k\in \mathbb{N}^*,
	\]
	and 
	\[
	u_{j}^{k,C}:= \max(u_j^k, V_{\theta^j}-C). 
	\]
	Observe that for $C,j$ fixed, the functions $u_{j}^{k,C}\geq V_{\theta^j}-C$ are uniformly bounded in $U$ (since $V_{\theta^j}$ is bounded in $U$) and converge in capacity to $u_j^{C}$ as $k\to +\infty$. Moreover, $f_{j}^{k,C,\vep}=0$ if $u_j^{k}\leq V_{\theta^j}-C$. By locality of the non-pluripolar product we can write 
	\[
	f^{k,C,\vep}  \chi \theta^1_{u^k_1} \wedge \ldots \wedge \theta^n_{u^k_n} =f^{k,C,\vep}  \chi \theta^1_{u^{k,C}_1} \wedge \ldots \wedge \theta^n_{u^{k,C}_n} ,
	\]
	where $f^{k,C,\vep}=f_1^{k,C,\vep}\cdots f_n^{k,C,\vep}$. 
For each $C,\vep$ fixed the functions $f^{k,C,\vep}$ are quasi-continuous,  uniformly bounded (with values in $[0,1]$) and converge in capacity to $f^{C,\vep}:=f_1^{C,\vep}\cdots f_n^{C,\vep}$, where $f_j^{C,\vep}$  is defined by 
	\[
	f_j^{C,\vep}  := \frac{\max(u_j-V_{\theta^j}+C,0)}{\max(u_j-V_{\theta^j}+C,0)+\vep}. 
	\]
With the information above we can apply Proposition \ref{prop: xing_conv}  to get that
	\[
	f^{k,C,\vep}  \chi \theta^1_{u^{k,C}_1} \wedge \ldots \wedge \theta^n_{u^{k,C}_n}  \rightarrow f^{C,\vep}  \chi \theta^1_{u^{C}_1} \wedge \ldots \wedge \theta^n_{u^{C}_n} \ \textrm{as}\ k\to +\infty,
	\]
	in the weak sense of measures on $U$. In particular since $0\leq f^{k,C,\vep}\leq 1$ we have that 
	\begin{eqnarray*}
		\liminf_{k\to+\infty} \int_X \chi \theta^1_{u^{k}_1} \wedge \ldots \wedge \theta^n_{u^{k}_n} & \geq &  \liminf_{k\to+\infty} \int_U f^{k,C,\vep}  \chi \theta^1_{u^{k,C}_1} \wedge \ldots \wedge \theta^n_{u^{k,C}_n}\\
		&\geq & \int_U f^{C,\vep}  \chi \theta^1_{u^{C}_1} \wedge \ldots \wedge \theta^n_{u^{C}_n}.
	\end{eqnarray*}
	Now, letting $\vep\to 0$ and then $C\to +\infty$, by definition of the non-pluripolar product we obtain 
	\[
		\liminf_{k\to+\infty} \int_X \chi \theta^1_{u^{k}_1} \wedge \ldots \wedge \theta^n_{u^{k}_n}   
		\geq  \int_U  \chi \theta^1_{u_1} \wedge \ldots \wedge \theta^n_{u_n}.
	\]
	Finally, letting $U$ increase to $\Omega$ and noting that the complement of $\Omega$ is pluripolar we conclude the proof of the first statement of the the theorem.  
	
To prove the last statement, we set $\mu_k:= \theta^1_{u^k_1} \wedge \ldots \wedge \theta^n_{u^k_n}$ and $\mu:= \theta^1_{u_1} \wedge \ldots \wedge \theta^n_{u_n}$.  Note that the total mass of these measures is bounded by $\int_X \theta^1\wedge \ldots \wedge \theta^n$ (\cite[Definition 1.17]{BEGZ10}).  As a result, by the Banach-Alaoglu theorem, it suffices to show that any cluster point of $\{\mu_k\}_k$ coincides with $\mu$.  Let $\nu$ be such a cluster point and assume (after extracting a subsequence) that $\mu_k$ converges weakly to $\nu$.  Condition \eqref{eq: global_mass_semi_cont} implies that $\nu(X)\leq \mu(X)$. It suffices to argue that  $\nu \geq \mu$, which is a consequence of the first statement, thus finishing the proof. 
\end{proof}

Now we move on to the monotonicity of non-pluripolar products:

\begin{theorem}\label{thm: BEGZ_monotonicity_full}
	Let $\theta^j, j \in \{1,\ldots,n\}$ be smooth closed real $(1,1)$-forms on $X$ whose cohomology classes are pseudoeffective. Let $u_j,v_j \in \textup{PSH}(X,\theta^j)$ be such that $u_j$ is less singular than $v_j$ for all $j \in \{1,\ldots,n\}$. Then 
	\[
	\int_X \theta^1_{u_1} \wedge \ldots \wedge \theta^n_{u_n}  \geq  \int_X \theta^1_{v_1} \wedge \ldots \wedge \theta^n_{v_n}. 
	\]
\end{theorem}

\begin{proof} By the same reason as in  Proposition \ref{prop: comparison generalization}, we can assume that the classes $\{\theta^j\}$ are in fact big. 
For each $t>0$ we set $v_{j}^{t}:= \max(u_{j}-t,v_j)$ for $j=1,...,n$. Observe that the $v_j^t$ converge decreasingly  to $v_j$ as $t \to \infty$. In particular, by \cite[Proposition 3.7]{GZ05} the convergence holds in capacity.  As $v_{j}^{t}$ and $u_j$ have the same singularity type, it follows from Proposition \ref{prop: comparison generalization} that 
\[
\int_X \theta^1_{u_1} \wedge \ldots \wedge  \theta^n_{u_n}=\int_X \theta^1_{v_1^t} \wedge \ldots \wedge  \theta^n_{v_n^t}.
\]
Letting $t \to \infty$, the first part of Theorem \ref{thm: lsc of MA measures} allows to conclude the argument.
\end{proof}

\begin{remark}\label{rem: increasing implies capacity} We note that condition \eqref{eq: global_mass_semi_cont} in Theorem \ref{thm: lsc of MA measures} is automatically satisfied if $u^k_j \nearrow u_j$ a.e. as $k \to \infty$. Indeed, in this case $u_j^k\to u_j$ in capacity (see \cite[Proposition 4.25]{GZ17}), and by Theorem \ref{thm: BEGZ_monotonicity_full} we have
$\int_ X \theta^1_{u_1} \wedge \ldots \wedge \theta^n_{u_n} \geq \limsup_k \int_X \theta^1_{u^k_1} \wedge \ldots \wedge \theta^n_{u^k_n}.$

On the other hand, if $u^k_j,u_j \in \mathcal E(X,\theta^j)$, by Corollary \ref{cor: full mass} below, it follows that \eqref{eq: global_mass_semi_cont} is again automatically satisfied. Moreover, in the next section we will show that this last property holds for potentials of relative full mass as well (see Corollary \ref{cor: convexity}), giving Theorem \ref{thm: BEGZ_monotonicity_full} a more broad spectrum of applications.
\end{remark}

\section{Pluripotential theory with relative full mass}\label{sec 3}

\subsection{Non-pluripolar products of relative full mass}

Suppose $\theta^j, \ j \in \{1,\ldots,n\}$ are smooth closed real $(1,1)$-forms on $X$ with $\{\theta^j\}$ pseudoeffective. Let $\phi_j,\psi_j \in \textup{PSH}(X,\theta^j)$, be such that $\phi_j$ is less singular than $\psi_j$. We say that $\theta^1_{\psi_1} \wedge\ldots\wedge\theta^n_{\psi_n}$ has full mass with respect to $\theta^1_{\phi_1} \wedge\ldots\wedge\theta^n_{\phi_n}$ (notation: $(\psi_1,\ldots,\psi_n) \in \mathcal E(X, \theta^1_{\phi_1},\ldots,\theta^n_{\phi_n})$) if 
$$\int_X \theta^1_{\psi_1} \wedge\ldots\wedge\theta^n_{\psi_n} = \int_X \theta^1_{\phi_1} \wedge\ldots\wedge\theta^n_{\phi_n}.$$
By Theorem \ref{thm: BEGZ_monotonicity_full}, in general we only have that the left hand side is less than the right hand side in the above identity. 

In the particular case when the potentials involved are from the same cohomology class $\{\theta\}$, and $\phi,\psi \in \textup{PSH}(X,\theta)$ with $\phi$ less singular than $\psi$ along with $\int_X\theta_\phi^n = \int_X \theta_\psi^n$, then we simply write $\psi \in \mathcal E(X,\theta,\phi)$, and say that $\psi$ has \emph{full mass relative to} $\theta^n_\phi$. When $\phi = V_\theta$, we recover the well known concept of full mass currents from the literature (see \cite{BEGZ10}).

As a consequence of Theorem \ref{thm: lsc of MA measures}, we prove a criterion for  testing membership in $\mathcal E(X, \theta^1_{\phi_1},\ldots,\theta^n_{\phi_n})$:

\begin{prop}\label{prop: env_pluri_general} Let $\theta^j, \ j \in \{1,\ldots,n\}$ be smooth closed real $(1,1)$-forms on $X$ with cohomology classes that are pseudoeffective. For all $j \in \{1,\ldots,n\}$  we choose $\phi_j,\psi_j \in \textup{PSH}(X,\theta^j)$ such that $\phi_j$ is less singular than $\psi_j$. If $P_{\theta^j}[\psi_j](\phi_j)=\phi_j$ then $(\psi_1,\ldots,\psi_n) \in \mathcal E(X, \theta^1_{\phi_1},\ldots,\theta^n_{\phi_n})$. 
\end{prop}

\begin{proof}If $P_{\theta^j}[\psi_j](\phi_j)=\phi_j$, then $v_j^C:= P_{\theta^j}(\psi_j + C,\phi_j) \nearrow \phi_j$ a.e., as $C \to \infty$. Theorem \ref{thm: lsc of MA measures} and Remark \ref{rem: increasing implies capacity} then imply that
$$\lim_{C\rightarrow +\infty} \int_X \theta_{v_1^C} \wedge ...\wedge \theta_{v_n^C} = \int_X  \theta_{\phi_1} \wedge ...\wedge \theta_{\phi_n} $$
As $P_{\theta^j}(\psi_j + C,\phi_j)$ has the same singularity type as $\psi_j$ for any $C$, the result follows from Proposition \ref{prop: comparison generalization}.
\end{proof}
As a result of this simple criterion, we obtain that  condition \eqref{eq: global_mass_semi_cont} in Theorem \ref{thm: lsc of MA measures} is satisfied if the potentials $u^k_j,u_j$ are from $\mathcal E(X,\theta^j)$:

\begin{coro}\label{cor: full mass}  Let $\theta^j, \ j \in \{1,\ldots,n\}$ be smooth closed real $(1,1)$-forms on $X$ with cohomology classes that are pseudoeffective.
If  $\psi_j\in \mathcal E(X, \theta^j), \ j\in \{1, \ldots, n\}$, then 
$$\int_X  \theta^1_{\psi_1} \wedge\ldots\wedge\theta^n_{\psi_n} = \int_X \theta^1_{V_{\theta^1}} \wedge \ldots  \wedge \theta_{V_{\theta^n}}^n,$$
or equivalently, $(\psi_1,\ldots,\psi_n) \in \mathcal E(X,\theta^1_{V_{\theta^1}},\ldots, \theta^n_{V_{\theta^n}})$.
\end{coro}
\begin{proof}
By \cite[Theorem 1.2]{DDL16} we have $P_{\theta^j}[\psi_j]:=P_{\theta^j}[\psi_j](V_{\theta^j})=V_{\theta^j}$. Hence Proposition \ref{prop: env_pluri_general} yields the conclusion.
\end{proof}

\begin{remark}\label{ex: zero mass}
Unfortunately, the reverse direction in Proposition \ref{prop: env_pluri_general} does not hold in general. Indeed, let $X = \mathbb{ C}\mathbb{ P}^1 \times \mathbb{ C}\mathbb{ P}^1$ with $\theta= \pi_1^* \omega_{FS}+\pi_2^\star \omega_{FS}$, where $\pi_1, \pi_2$  are the  projections to the first and second component respectively. 

Consider $\phi(z,w):=u(z) +v(w) \in \textup{PSH}(X,\theta)$ where $u, v\leq 0$ satisfy $\omega_{FS}+i\ddbar u =\delta_{z_0}$ and $\omega_{FS}+i\ddbar v =\delta_{w_0}$, where $\delta_{z_0},\delta_{w_0}$ are Dirac masses for some $z_0,w_0 \in \Bbb{C}\Bbb{P}^1$.  Clearly, $\int_X \theta_{\phi}^2= \int_X \theta_{\pi^*_2 v}^2=0$, and since $\phi\leq \pi_2^*v$, we have that $\phi\in \mathcal{E}(X ,\theta, \pi_2^*v)$. 

On the other hand, we know that $\phi$ has the same Lelong numbers as $P_\theta[\phi]$ (\cite[Theorem 1.1]{DDL16}). As  $P_{\theta}[\phi](\pi_2^*v) \leq P_{\theta}[\phi]$, it follows however that $P_{\theta}[\phi](\pi_2^*v) \lneq \pi_2^*v$, since at some points of $\mathbb{ C}\mathbb{ P}^1 \times \mathbb{ C}\mathbb{ P}^1$ the Lelong number of $\pi_2^*v$ is zero, but the Lelong number of $\phi$ is non-zero. 

As we will see below (Theorem \ref{thm: E_memb_char}), a partial converse of Proposition \ref{prop: env_pluri_general} is still possible under the assumption of non-vanishing total mass.
\end{remark}

In the remaining part of this subsection we prove basic properties of non-pluripolar products with relative full mass, that will be used later in this work.

\begin{lemma}\label{lem: E_def_lem} Suppose $\phi_j, \psi_j\in  \textup{PSH}(X,\theta^j)$. Then $(\psi_1, \ldots, \psi_n) \in \mathcal E(X,\theta^1_{\phi_1}, \ldots , \theta^n_{\phi_n} )$ if and only if $\phi_j$ is less singular than $\psi_j$, and $\int_{\bigcup_j\{\psi_j\leq  \phi_j -k \}}\theta^1_{\max(\psi_1,\phi_1 - k)} \wedge \ldots \wedge\theta^n_{\max(\psi_n,\phi_n - k)} \to 0$, as $k \to \infty$.
\end{lemma}
\begin{proof}If $\phi_j$ is less singular than $\psi_j$, then $\max(\psi_j,\phi_j-k)$ has the same singularity type as $\phi_j$. Consequently, Proposition \ref{prop: comparison generalization} gives that
\begin{flalign*}
\int_X \theta^1_{\phi_1} & \wedge \ldots \wedge  \theta^n_{\phi_n} = \int_X \theta^1_{\max(\psi_1,\phi_1 - k)} \wedge \ldots \wedge\theta^n_{\max(\psi_n,\phi_n - k)}\\
&=\int_{\cap_j\{\psi_j > \phi_j - k\}}\theta^1_{\psi_1} \wedge \ldots \wedge  \theta^n_{\psi_n} + \int_{\bigcup_j\{\psi_j\leq  \phi_j -k \}}\theta^1_{\max(\psi_1,\phi_1 - k)} \wedge \ldots \wedge\theta^n_{\max(\psi_n,\phi_n - k)}.
\end{flalign*}
Since $\int_{\cap_j\{\psi_j > \phi_j - k\}}\theta^1_{\psi_1} \wedge \ldots \wedge  \theta^n_{\psi_n} \to \int_X \theta^1_{\psi_1}  \wedge \ldots \wedge  \theta^n_{\psi_n}$ as $k \to \infty$, the equivalence of the lemma follows after we take the limit $k \to \infty$ in the above identity.
\end{proof}

As a consequence of this last lemma and the locality of the non-pluripolar product with respect to the plurifine topopolgy we obtain the following uniform estimate
\begin{flalign*}
\lim_{k \to \infty}\bigg|\int_B & \theta^1_{\psi_1} \wedge \ldots \wedge \theta^n_{\psi_n} - \int_B \theta^1_{\max(\psi_1,\phi_1 - k)} \wedge \ldots \wedge\theta^n_{\max(\psi_n,\phi_n - k)}\bigg| \leq  \nonumber \\
&\leq 2 \int_{\bigcup_j\{\psi_j\leq  \phi_j -k \}}\theta^1_{\max(\psi_1,\phi_1 - k)} \wedge \ldots \wedge\theta^n_{\max(\psi_n,\phi_n - k)} \to 0. 
\end{flalign*}
for any Borel set $B \subset X$ and $(\psi_1, \ldots, \psi_n) \in \mathcal E(X,\theta^1_{\phi_1}, \ldots , \theta^n_{\phi_n} )$.

Lastly, we note the \emph{partial comparison principle} for non-pluripolar products of relative full mass, generalizing a result of Dinew from \cite{Dw09b}:

\begin{prop}\label{prop: general CP} Suppose $\phi_k,\psi_k \in \textup{PSH}(X,\theta^k), k=1,\cdots j\leq n$ and $\phi\in {\rm PSH}(X,\theta)$. Assume that $(u,...,u,\psi_1,...,\psi_j), (v,...,v,\psi_1,...,\psi_j) \in \mathcal{E}(X,\theta_{\phi},...,\theta_{\phi}, \theta_{\phi_1},...,\theta_{\phi_j})$. Then 
\[
\int_{\{u<v\}} \theta_{v}^{n-j}\wedge  \theta^1_{\psi_1} \wedge \ldots \wedge \theta^j_{\psi_j} \leq \int_{\{u<v\}} \theta_{u}^{n-j}\wedge  \theta^1_{\psi_1} \wedge \ldots \wedge \theta^j_{\psi_j}.
\] 
\end{prop}

\begin{proof} The proof follows the argument of \cite[Proposition 2.2]{BEGZ10} with a vital ingredient from Theorem \ref{thm: BEGZ_monotonicity_full}.

Since $\max(u,v)$ is more singular than $\phi$, and $\psi_k$ is more singular than $\phi_k$, for $k=1,...,j$, it follows from the assumption and Theorem \ref{thm: BEGZ_monotonicity_full} that
\begin{eqnarray*}
\int_X \theta_{\phi}^{n-j}\wedge  \theta^1_{\phi_1} \wedge \ldots \wedge \theta^j_{\phi_j} &=&
\int_X \theta_{v}^{n-j}\wedge  \theta^1_{\psi_1} \wedge \ldots \wedge \theta^j_{\psi_j} \\
&\leq & \int_X \theta_{\max(u,v)}^{n-j}\wedge  \theta^1_{\psi_1} \wedge \ldots \wedge \theta^j_{\psi_j}\\
&\leq &   \int_X \theta_{\phi}^{n-j}\wedge  \theta^1_{\phi_1} \wedge \ldots \wedge \theta^j_{\phi_j}. 
\end{eqnarray*}
Hence the inequalities above are in fact equalities. 
By locality of the non-pluripolar product we then can write:
\begin{flalign*}
	\int_X \theta_{\max(u,v)}^{n-j} \wedge \theta^1_{\psi_1} \wedge ... \wedge \theta^j_{\psi_j}
&\geq  \int_{\{u> v\}} \theta_{u}^{n-j} \wedge \theta^1_{\psi_1} \wedge ... \wedge \theta^j_{\psi_j} + \int_{\{v> u\}} \theta_{v}^{n-j} \wedge \theta^1_{\psi_1} \wedge ... \wedge \theta^j_{\psi_j} \\
&=  \int_X \theta_{u}^{n-j} \wedge \theta^1_{\psi_1} \wedge ... \wedge \theta^j_{\psi_j} - \int_{\{u\leq  v\}} \theta_{u}^{n-j} \wedge \theta^1_{\psi_1} \wedge ... \wedge \theta^j_{\psi_j} \\
&+  \int_{\{v> u\}} \theta_{v}^{n-j} \wedge \theta^1_{\psi_1} \wedge ... \wedge \theta^j_{\psi_j}\\
& =  \int_X \theta_{\max(u,v)}^{n-j} \wedge \theta^1_{\psi_1} \wedge ... \wedge \theta^j_{\psi_j} - \int_{\{u\leq  v\}} \theta_{u}^{n-j} \wedge \theta^1_{\psi_1} \wedge ... \wedge \theta^j_{\psi_j} \\
&+  \int_{\{v> u\}} \theta_{v}^{n-j} \wedge \theta^1_{\psi_1} \wedge ... \wedge \theta^j_{\psi_j}.
\end{flalign*}
We thus get
\begin{flalign*}
\int_{\{u < v\}} \theta_{v}^{n-j} \wedge \theta^1_{\psi_1} \wedge \ldots \wedge \theta^j_{\psi_j}  \leq \int_{\{u \leq v\}} \theta_{u}^{n-j} \wedge \theta^1_{\psi_1} \wedge \ldots \wedge \theta^j_{\psi_j}.
\end{flalign*}
Replacing $u$ with $u + \varepsilon$ in the above inequality, and letting $\varepsilon \searrow 0$, by the monotone convergence theorem we arrive at the result.
\end{proof}

In the next subsection, after we explore the class $\mathcal E(X,\theta,\phi)$, we will give a partial comparison principle specifically for this class, as a corollary of the above general proposition. Here we only note the following trivial consequence:

\begin{coro}\label{comparison principle} Suppose $\phi \in \textup{PSH}(X,\theta)$ and assume that $u,v\in \Ec(X,\theta,\phi)$. Then 
\[
\int_{\{u<v\}} \theta_{v}^{n} \leq \int_{\{u<v\}} \theta_{u}^{n}.\] 
\end{coro}

Note that the above result generalizes \cite[Corollary 2.3]{BEGZ10}.
\subsection{The envelope $P_\theta[\phi]$ and the class $\mathcal E(X,\theta,\phi)$} 

Let $\theta$ be a smooth closed real $(1,1)$-form on $X$ which represents a big class and  fix  $\phi\in \psh(X,\theta)$ such that $\phi\leq 0$. In this short subsection we focus on the relative full mass class $\mathcal E(X,\theta,\phi)$. 

Based on our previous findings, one wonders if the following set of potentials has a maximal element:

\[ F_\phi:= \bigg\{v\in \psh(X,\theta) \setdef  \phi \leq v \leq 0 \ \textrm{and} \ \int_X \theta_v^n =\int_X \theta_{\phi}^n\bigg\}. 
\]
In other words, does there exist a \emph{least singular} potential that is less singular than $\phi$ but has the same full mass as $\phi$. As we will see, if $\int_X \theta_\phi^n >0$, this is indeed the  case, moreover this maximal potential is equal to $P_\theta[\phi]$ (Theorem \ref{thm: ceiling coincide envelope non collapsing}). 

Linking the envelope $P_\theta[\phi]$ to the class $\mathcal E(X,\theta,\phi)$, observe that $\phi \leq P_\theta[\phi] \leq 0$ and  $\int_X \theta_{P_\theta[\phi]}^n= \int_X  \theta_{\phi}^n$, in particular $P_\theta[\phi] \in F_\phi$ and $\phi \in \mathcal E(X, \theta, P_{\theta}[\phi])$. Indeed, since $P_{\theta}(\phi+C,0) \nearrow P_\theta[\phi](0)=P_{\theta}[\phi]$  a.e. as $C \to +\infty$, using Theorem \ref{thm: BEGZ_monotonicity_full} and Theorem \ref{thm: lsc of MA measures} we can conclude that $\int_X \theta_{P_\theta[\phi]}^n= \int_X  \theta_{\phi}^n$.

In our study, we will need the following preliminary result, providing an estimate for the complex Monge-Amp\`ere operator of rooftop envelopes, that builds on recent progress in \cite{GLZ17}:
\begin{lemma}\label{lem: GLZ non-full}
	Let $\varphi,\psi\in \psh(X,\theta)$. If $P_{\theta}(\varphi,\psi) \neq -\infty$ then
	\[
	\theta_{P_{\theta}(\varphi,\psi)}^n \leq\mathbbm{1}_{\{P_{\theta}(\varphi,\psi)=\varphi\}} \theta_{\varphi}^n + \mathbbm{1}_{\{P_{\theta}(\varphi,\psi)=\psi\}} \theta_{\psi}^n. 
	\]
\end{lemma}

\begin{proof}
For each $t>0$ we set $\varphi_t:=\max(\varphi,V_{\theta}-t), \psi_t:=\max(\psi,V_{\theta}-t)$ and $v_t:= P_{\theta}(\varphi_t,\psi_t)$. Set $v:=P_\theta(\phi, \psi)$.  Since $\varphi_t, \psi_t$ have minimal singularities, it follows from \cite[Lemma 4.1]{GLZ17} that 
\begin{equation}\label{eq: GLZ17est}
	\theta_{v_t}^n \leq \mathbbm{1}_{\{v_t=\varphi_t\}} \theta_{\varphi_t}^n + \mathbbm{1}_{\{v_t=\psi_t\}}\theta_{\psi_t}^n.
\end{equation}
For $C>0$ we introduce 
\[
G_C:= \{v>V_{\theta}-C\}, \ \ v^C:=\max(v,V_{\theta}-C), \ \textup{ and } \ v_{t}^C:=\max(v_t,V_{\theta}-C).
\]
Since ${P_{\theta}(\varphi,\psi)}\leq \varphi,\psi,v_t$, we have $G_C \subset \{V_\theta -C < \varphi\}\cap \{V_\theta -C < \psi\} \cap \{V_\theta -C < v_t\}$. 
For arbitrary $A>0$ and $t > C$, this inclusion allows to build on \eqref{eq: GLZ17est} and write:
\begin{eqnarray}
\nonumber		\id_{G_C}\theta^n_{v^C_t}= \id_{G_C}\theta^n_{v_t}& \leq &  \id_{\{v_t=\varphi_t\}\cap G_{C}} \theta_{\varphi_t} ^n+ \id_{\{v_t=\psi_t\}\cap G_{C}}\theta_{\psi_t}^n\\
\nonumber	& \leq & \id_{\{v_t=\varphi_t\} \cap \{\varphi>V_{\theta}-t\}} \theta_{\varphi_t}^n + \id_{\{v_t=\psi_t\}\cap \{\psi>V_{\theta}-t\}}\theta_{\psi_t}^n\\
\nonumber	& = & \id_{\{v_t=\varphi_t\} \cap \{\varphi>V_{\theta}-t\}} \theta_{\varphi}^n + \id_{\{v_t=\psi_t\}\cap \{\psi>V_{\theta}-t\}}\theta_{\psi}^n\\
    &\leq &   e^{A(v_t-\varphi_t)}\theta_{\varphi}^n +  e^{A(v_t-\psi_t)}\theta_{\psi}^n. \label{eq: interm_est}
	\end{eqnarray}
To proceed, we want to prove that 
\begin{equation}\label{eq: quasi open}
\liminf_{t\to \infty} \mathbbm{1}_{G_C} \theta^n_{v_t^C} \geq \id_{G_C}  \theta^n_{v^C}. 
\end{equation}
More precisely, alluding to the Banach-Alaoglu theorem, we want to show that any weak limit of $\big\{ \mathbbm{1}_{G_C} \theta^n_{v_t^C}\big\}_t$ is greater than $\mathbbm{1}_{G_C}\theta^n_{v^C}$.

Let $U := \textup{Amp}(\theta)$. The potential $V_\theta$ is locally bounded on $U$, hence so is $v^C_t$ and $v^C$. To obtain \eqref{eq: quasi open}, we employ an idea from the proof of  Theorem \ref{thm: lsc of MA measures}.  For $\vep>0$ consider 
\[
f_{\vep}:= \frac{\max(v-V_{\theta}+C,0)}{\max(v-V_{\theta}+C,0) +\vep},
\] 
and observe that $f_{\vep} \geq 0$ is quasicontinuous on $X$. Moreover, $f_{\vep}$ increase pointwise to $\mathbbm{1}_{G_C}$ as $\varepsilon$ goes to zero. Since $v_t^C \searrow v^C$ as $t \to \infty$, from \cite[Theorem 4.26]{GZ17} it follows that  $f_\varepsilon \theta_{v_t^C}^n\big|_{U} \to f_\varepsilon\theta_{v^C}^n\big|_{U}$ weakly. Using this we can write
\[
\liminf_{t\to \infty} \mathbbm{1}_{G_C} \theta^n_{v_t^C} \big|_{U}\geq \lim_{t\to+\infty}  f_{\vep}\theta^n_{v_t^C}\big|_{U} =   f_{\vep}  \theta^n_{v^C}\big|_{U}.
\]
Since $X\setminus U$ is pluripolar, we let $\vep\to 0$ and use the monotone convergence theorem to conclude \eqref{eq: quasi open}. 

Now, letting $t\to \infty$ in \eqref{eq: interm_est}, the estimate in \eqref{eq: quasi open} allows to conclude that:
\[
\id_{G_C} \theta_{\max(P_\theta(\varphi,\psi),V_\theta - C)}^n \leq e^{A(P_\theta(\varphi,\psi)-\varphi)} \theta_{\varphi}^n + e^{A(P_\theta(\varphi,\psi)-\psi)} \theta_{\psi}^n.
\]
Letting $C\to \infty$, and later $A\to \infty$, we arrive at the conclusion.    
\end{proof}

We prove in the following that the non-pluripolar complex Monge-Amp\`ere measure of  $P_{\theta}[\psi](\chi)$ has bounded density with respect to $\theta_\chi^n$.  This plays a crucial role in the sequel.

\begin{theorem}\label{thm: MA measure of WN envelope}
	Let $\psi,\chi \in \textup{PSH}(X,\theta)$, such that $\psi$ is more singular than $\chi$. Then  $\theta_{P_\theta[\psi](\chi)}^n \leq \mathbbm{1}_{\{P_\theta[\psi](\chi)=\chi\}} \theta_{\chi}^n$. In particular, $\theta^n_{P_\theta[\psi]} \leq \mathbbm{1}_{\{P_\theta[\psi]=0\}}\theta^n$.
\end{theorem}

This result can be thought of as a regularity result for the envelope $P_{\theta}[\psi](\chi)$. For a more precise regularity result on such envelopes in the particular case of potentials with algebraic singularities we refer to \cite[Theorem 1.1]{RWN2}.

\begin{proof}
	Without loss of generality we can assume that $\psi ,\chi\leq 0$. For each $t>0$ we consider $P_{\theta}(\psi+t,\chi)$. Since $\psi$ is more singular than $\chi$, we note that $P_{\theta}(\psi+t,\chi)$ has the same singularity type as $\psi$ and $P_{\theta}(\psi+t,\chi) \nearrow P_\theta[\psi](\chi)$ a.e.. It follows from Lemma \ref{lem: GLZ non-full} that 
    \[
    \theta_{P_{\theta}(\psi+t,\chi)}^n \leq \mathbbm{1}_{\{P_{\theta}(\psi+t,\chi)=\psi+t\}} \theta_{\psi}^n + \mathbbm{1}_{\{P_{\theta}(\psi+t,\chi)=\chi\}} \theta_{\chi}^n.  
    \]
    Since $\{P_{\theta}(\psi+t,\chi)=\psi+t\}\subset \{\psi + t\leq \chi\}\subset \{\psi + t\leq V_{\theta}\}$, and the latter decreases to a pluripolar set, the first term on the right-hand side above goes to zero, as $t\to \infty$. For the second term, we observe that $\{P_{\theta}(\psi+t,\chi)=\chi\}\subset \{P_\theta[\psi](\chi)=\chi\}$. Hence applying Theorem \ref{thm: lsc of MA measures} the result follows. 
    
For the last statement, we can apply the above argument to $\chi:=V_{\theta}$, and note that  from \cite[(1.2)]{Ber13} (see also  \cite[Theorem 2.6 (arXiv version)]{DDL16},  \cite[Proposition 5.2]{GLZ17}) it follows that $\theta_{V_{\theta}}^n \leq \mathbbm{1}_{\{V_{\theta}=0\}}\theta^n$.
\end{proof}

Using the above result, we can establish a \emph{non-collapsing property} for the class of potentials with the same singularity type as $\phi$, when $\theta_{\phi}^n(X)>0$:
\begin{coro}
\label{cor: non-collapsing}
Assume that $\phi \in \psh(X,\theta)$ is such that $\int_X \theta_{\phi}^n>0$. If $U$ is a Borel subset of $X$  with positive Lebesgue measure, then there exists $\psi\in \psh(X,\theta)$  having the same singularity type as $\phi$  such that $\theta_{\psi}^n(U)>0$.  
\end{coro}
\begin{proof}
It follows from \cite[Theorem A,B]{BEGZ10} that  there exists $h\in \psh(X,\theta)$ with minimal singularities such that $\theta_h^n=c\mathbbm{1}_U \omega^n$, for some normalization constant $c>0$. For  $C>0$ consider $\varphi_C:= P_{\theta}(\phi+C, h)$ and note that $\varphi_C$ has the same singularities as $\phi$. It follows from Lemma \ref{lem: GLZ non-full} that 
\[
\theta_{\varphi_C}^n \leq \mathbbm{1}_{\{\varphi_C=\phi+C\}} \theta_{\phi}^n + \mathbbm{1}_{\{\varphi_C=h\}} \theta_{h}^n \leq \mathbbm{1}_{\{\phi+C \leq h\}} \theta_{\phi}^n + c\mathbbm{1}_{\{\varphi_C=h\} \cap U} \omega^n.
\]
Since $\theta_{\phi}^n$ is non-pluripolar, we have that $\lim_{C\to \infty}  \int_{\{\phi+C \leq h\}}\theta_{\phi}^n=0$. Thus  for $C>0$ big enough, by the above estimate we have that 
\[
\int_{X\setminus U} \theta_{\varphi_C}^n \leq  \int_{\{\phi+C \leq h\}} \theta_{\phi}^n < \int_X \theta_{\varphi_C}^n,
\]
where in the last inequality we used the fact that $ \int_X \theta_{\varphi_C}^n= \int_X \theta_{\phi}^n>0$.
This implies  that $\int_{U}\theta_{\varphi_C}^n>0$ for big enough $C>0$, finishing the argument.
\end{proof}

A combination of Corollary \ref{cor: non-collapsing} and \cite[Corollary 4.2]{WN17} immediately gives the following  version of the domination principle, making the conclusion of \cite[Corollary 4.2]{WN17} more precise: 
\begin{coro}
\label{cor: general domination principle} 
Assume that $u,v\in \psh(X,\theta)$, $u$ is less singular than $v$ and $\int_X \theta_u^n>0$. If $u\geq v$ a.e. with respect to $\theta_u^n$, then $u\geq v$ on $X$. 
\end{coro}

\begin{proof}
Assume by contradiction that $\{u<v\}\subseteq X$ has positive Lebesgue measure. Then, by Corollary \ref{cor: non-collapsing} we can ensure that there exists $\psi\in \psh(X,\theta)$  having the same singularity type as $u$  such that $\theta_{\psi}^n(\{u<v\})>0$. On the other hand, since $\theta_{u}^n(\{u<v\})=0$, \cite[Corollary 4.2]{WN17} gives that $\theta_{\psi}^n(\{u<v\})=0$, which is a contradiction.
\end{proof}
The non-collapsing mass condition $\int_X \theta_u^n>0$ is trivially seen to be necessary.  
We now give the version of the \emph{domination principle} for the relative full mass class $\Ec(X,\theta,\phi)$:

\begin{prop}\label{prop: domination principle}
Suppose that $\phi\in \psh(X,\theta)$  satisfies $\int_X \theta_\phi^n >0$ and  $u,v\in \Ec(X,\theta,\phi)$. If $\theta_u^n(\{u<v\})=0$  then $u\geq v$. 
\end{prop}
\begin{proof} First, assume that $v$ is less singular than $u$. In view of Corollary \ref{cor: non-collapsing} it suffices to prove that $\theta_{h}^n(\{u<v\})=0$ for all $h\in \psh(X,\theta)$ with the same singularity type as $u$. Let $h$ be such a potential, and after possibly adding a constant, we can assume that $h\leq u,v$.  We claim that for each $t\in (0,1)$, $(1-t)v+th\in \Ec(X,\theta,\phi)$. Indeed, since $(1-t)v + th$ is less singular than $u$, and more singular than $v$, by Theorem \ref{thm: BEGZ_monotonicity_full} we can write
$$\int_X \theta_u^n \leq \int_X \theta_{(1-t)v+th}^n \leq \int_X \theta_v^n.$$
The comparison principle (Corollary \ref{comparison principle}) allows then to write:
$$t^n\int_{\{u< (1-t) v + t h\}}\theta_{h}^n \leq \int_{\{u < (1-t)v+th\}}\theta_{(1-t)v+th}^n \leq \int_{\{u < v\}}\theta_u^n=0. $$
Since $0=\theta_h^n(\{u<(1-t)v + th\}) \nearrow \theta_h^n(\{u<v\})$, as $t \to 0$, it follows that $\theta_h^n(\{u<v\})=0$. 

For the general case, we observe that $\theta_u^n(\{u<v\})= \theta_u^n (\{u< \max(u,v)\})$, and the first step implies $u\geq \max(u,v) \geq v$.
\end{proof}

Next we show that $F_\phi$, the set of potentials introduced in the beginning of this subsection, has a very specific maximal element:

\begin{theorem}\label{thm: ceiling coincide envelope non collapsing}
Assume that   $\phi\in \psh(X,\theta)$ satisfies $\int_X \theta_{\phi}^n>0$ and $\phi \leq 0$. Then $$P_{\theta}[\phi] = \sup_{v \in F_\phi} v.$$
In particular, $P_{\theta}[\phi]=P_{\theta}[P_{\theta}[\phi]]$.
\end{theorem}

As remarked in the beginning of the subsection, $P_{\theta}[\phi] \in F_\phi$, hence by the above result $P_{\theta}[\phi]$ is the maximal element of $F_\phi$.

\begin{proof}
Let $u\in F_{\phi}$. By Theorem \ref{thm: MA measure of WN envelope} we have
\begin{flalign*}
\theta_{P_{\theta}[\phi]}^n (\{P_{\theta}[\phi] < u\}) &\leq \mathbbm{1}_{\{P_{\theta}[\phi]=0\}}\theta^n  (\{P_{\theta}[\phi] < u\}) \leq \mathbbm{1}_{\{P_{\theta}[\phi]=0\}}\theta^n  (\{P_{\theta}[\phi] < 0\})=0.
\end{flalign*}
As $\phi \leq u$, and $\int_X \theta_\phi^n = \int_X \theta_u^n$,  by Theorem \ref{thm: BEGZ_monotonicity_full} and Theorem \ref{thm: lsc of MA measures} we have 
$$
\int_X \theta^n_{P_{\theta}[\phi]}=\int_X \theta_\phi^n = \int_X \theta_u^n=\int_X \theta^n_{P_{\theta}[u]}>0.
$$
Consequently, $P_{\theta}[\phi],u \in \mathcal E(X,\theta, P_{\theta}[u])$ and Proposition \ref{prop: domination principle} now insures that $P_{\theta}[\phi] \geq u$, hence $P_{\theta}[\phi]\geq \sup_{v \in F_{\phi}} v$. As $P_{\theta}[\phi] \in F_{\phi}$, it follows that $P_{\theta}[\phi]= \sup_{v \in F_{\phi}} v$. 

For the last statement notice that $P_\theta[\phi]=\sup_{v \in F_\phi} v \geq  \sup_{v \in F_{P_{\theta}[\phi]}} v=P_\theta[P_\theta[\phi]]$, since $F_\phi \supset F_{P_{\theta}[\phi]}$. The reverse inequality is trivial. 
\end{proof}

\begin{remark} The assumption $\int_X \theta_\phi^n >0$ is necessary in the above theorem. Indeed, in the setting of Remark \ref{ex: zero mass}, it can be seen that $P_{\theta}[\phi] \lneq \sup_{h \in F_\phi}h$, as the potential on the right hand side is greater than $\pi_2 ^*v$, since $\pi_2 ^*v \in F_\phi$.
\end{remark}

As a consequence of this last result, we obtain the following characterization of membership in $\mathcal E(X,\theta,\phi)$, providing a partial converse to Proposition \ref{prop: env_pluri_general}:

\begin{theorem}\label{thm: E_memb_char} Suppose $\phi \in \textup{PSH}(X,\theta)$ with $\int_X \theta_\phi^n >0$ and $\phi \leq 0$. The following are equivalent:\\
\noindent (i) $u \in \mathcal E(X,\theta,\phi)$.\\
\noindent (ii) $\phi$ is less singular than $u$, and $P_{\theta}[u](\phi)=\phi$.\\
\noindent (iii) $\phi$ is less singular than $u$, and $P_{\theta}[u]=P_{\theta}[\phi]$.
\end{theorem}

As a consequence of the equivalence between (i) and (iii), we see that the potential $P_{\theta}[u]$ stays the same for all $u \in \mathcal E(X,\theta,\phi)$, i.e., it is an invariant of this class. In particular, since $\mathcal E(X,\theta,\phi) \subset \mathcal E(X,\theta,P_{\theta}[\phi])$, by the last statement of Theorem \ref{thm: ceiling coincide envelope non collapsing}, it seems natural to only consider potentials $\phi$ that are in the image of the operator $\psi \to P_{\theta}[\psi]$, when studying classes of relative full mass $\mathcal E(X,\theta,\phi)$. What is more, in the next section it will be clear that considering such $\phi$ is not just more natural, but also necessary when trying to solve complex Monge-Amp\`ere equations with prescribed singularity.

\begin{proof}Assume that (i) holds. By Theorem \ref{thm: MA measure of WN envelope} it follows that $P_{\theta}[u](\phi) \geq \phi$ a.e. with respect to $\theta^n_{P_{\theta}[u](\phi)}$. Proposition \ref{prop: domination principle} gives $P_{\theta}[u](\phi) = \phi$, hence (ii) holds. 

Suppose (ii) holds. We can assume that $u \leq \phi \leq 0$. Then $P_{\theta}[u] \geq P_{\theta}[u](\phi) = \phi$. By the last statement of the previous theorem, this implies that 
$$P_{\theta}[u] = P_{\theta}[P_{\theta}[u]] \geq P_{\theta}[\phi].$$
As the reverse inequality is trivial, (iii) follows.

Lastly, assume that (iii) holds. By Theorem \ref{thm: BEGZ_monotonicity_full} and Theorem \ref{thm: lsc of MA measures} it follows that $\int_X \theta_u^n=\int_X \theta_{P_{\theta}[u]}^n=\int_X \theta_{P_{\theta}[\phi]}^n=\int_X \theta_\phi^n$, hence (i) holds.
\end{proof}

\begin{coro}\label{cor: convexity}
	Suppose $\phi \in \textup{PSH}(X,\theta)$ such that $\int_X \theta_{\phi}^n>0$. Then $\Ec(X,\theta,\phi)$ is convex. Moreover, given $\psi_1, \ldots, \psi_n \in \Ec(X, \theta, \phi)$ we have
   \begin{equation}\label{mixed mass} 
    \int_X \theta_{\psi_1}^{s_1}\wedge \ldots \wedge \theta_{\psi_n}^{s_n}= \int_X \theta_\phi^n,
    \end{equation}
 where $s_j \geq 0$ are integers such that $\sum_{j=1}^n s_j=n$.
\end{coro}
\begin{proof}
  Let $u,v\in \Ec(X,\theta,\phi)$ and fix $t\in (0,1)$. It follows from Theorem \ref{thm: E_memb_char} that $P_\theta[v](\phi)=P_\theta[u](\phi)=\phi$. This implies that  $$P_\theta[tv + (1-t)u](\phi)\geq tP_\theta[v](\phi) + (1-t)P_\theta[u](\phi)=\phi.$$
As the reverse inequality is trivial, another application of Theorem  \ref{thm: E_memb_char} gives that $t v + (1-t)u \in \mathcal E(X,\theta,\phi)$. 

We now prove the last statement. Since $\mathcal E(X,\theta,\phi)$ is convex, given $\psi_1, \ldots, \psi_n \in \Ec(X, \theta, \phi)$ we know that any convex combination $\psi:=\sum_{j=1}^n s_j \psi_j$ with $0\leq s_j\leq 1 $ and $\sum_j s_j=n$, belongs to $\Ec(X, \theta, \phi)$. Hence
$$\int_X \bigg( \sum_j s_j \theta_{\psi_j}\bigg)^n=\int_X \theta_{\psi}^n= \int_X\theta_\phi^n=\int_X \bigg( \sum_j s_j\theta_\phi \bigg)^n.$$
As a result, we have an identity of two homogeneous polynomials of degree $n$. Therefore all the coefficients of these polynomials have to be equal, giving \eqref{mixed mass}.
\end{proof}

Lastly, we provide another corollary, in the spirit of the partial comparison principle from Proposition \ref{prop: general CP}:

\begin{coro}\label{partial comparison principle} Suppose $\phi \in \textup{PSH}(X,\theta)$ with $\int_X \theta_{\phi}^n > 0$. Assume that $u,v,\psi_1,...,\psi_{j}\in \Ec(X,\theta,\phi)$ for some $j \in \{0,\ldots,n\}$. Then 
\[
\int_{\{u<v\}} \theta_{v}^{n-j}\wedge  \theta_{\psi_1} \wedge \ldots \wedge \theta_{\psi_j} \leq \int_{\{u<v\}} \theta_{u}^{n-j}\wedge  \theta_{\psi_1} \wedge \ldots \wedge \theta_{\psi_j}.\] 
\end{coro}
\begin{proof} The conclusion follows immediately from \eqref{mixed mass} together with Proposition \ref{prop: general CP}.
\end{proof}

\section{Complex Monge-Amp\`ere equations with prescribed singularity type}\label{sec 4}

Let $\theta$ be a smooth closed real $(1,1)$-form on $X$ such that $\{\theta\}$ is big and $\phi \in \textup{PSH}(X,\theta)$. By $\mathrm{PSH}(X,\theta,\phi)$  we denote the set of $\theta$-psh functions that are more singular than $\phi$. We say that $v \in \textup{PSH}(X,\theta,\phi)$ has \emph{relatively minimal singularities} if $v$ has the same singularity type as $\phi$. Clearly, $\mathcal E(X,\theta,\phi) \subset \textup{PSH}(X,\theta,\phi)$. 

Let $\mu$ be a non-pluripolar positive  measure on $X$ such that $\mu(X)=\int_X \theta_\phi^n > 0$.
Our aim is to study existence and uniqueness of solutions to the following equation of complex Monge-Amp\`ere type:
\begin{equation}\label{eq: CMAE_sing}
\theta_\psi^n = \mu, \ \ \ \psi \in \mathcal E(X,\theta,\phi).
\end{equation}
It is not hard to see that this equation does not have a solution for arbitrary $\phi$. Indeed, suppose for the moment that $\theta = \omega$, and  choose $\phi \in \mathcal E(X,\omega) :=\mathcal E(X,\omega,0)$ unbounded. It is clear that $\mathcal E(X,\omega,\phi) \subsetneq \mathcal E(X,\omega,0)$. By \cite[Theorem A]{BEGZ10}, the (trivial) equation $\omega_{\psi}^n = \omega^n, \ \psi \in \mathcal E(X,\omega,0)$ is \emph{only} solved by potentials $\psi$ that are constant over $X$, hence we cannot have $\psi \notin \mathcal E(X,\omega,\phi)$.

This simple example suggests that we need to be more selective in our choice of $\phi$, to make \eqref{eq: CMAE_sing} well posed. As it turns out, the natural choice is to take $\phi$ such that  $P_{\theta}[\phi]=\phi$, as suggested by our study of currents of relative full mass in the previous subsection. Therefore, for the rest of this section we ask that $\phi$ additionally satisfies:
\begin{equation}\label{eq: fixed_point}
\phi=P_{\theta}[\phi].
\end{equation}
Such a potential $\phi$  is called a model potential, and $[\phi]$ is a \emph{model type singularity}. As $V_\theta = P_\theta[V_\theta]$, one can think of such $\phi$ as generalizations of $V_\theta$, the potential with minimal singularity from \cite{BEGZ10}. Wee refer to Remark \ref{rem: examples} for natural constructions of model type singularities.

As a technical assumption, we will ask that $\phi$ has additionally \emph{small unbounded locus}, i.e., $\phi$ is locally  bounded outside a closed pluripolar set $A \subset X$. This will be needed to carry out arguments involving integration by parts in the spirit of \cite{BEGZ10}.

One wonders if maybe model type potentials (those that satisfy \eqref{eq: fixed_point}) always have small unbounded locus. Sadly, this is not the case, as the following simple example shows. Suppose $\theta$ is a K\"ahler form, and $\{x_j\}_j \subset X$ is a dense countable subset. Also let $v_j \in \textup{PSH}(X,\theta)$ such that $v_j < 0$, $\int_X v_j \theta^n = 1$, and $v_j$ has a positive Lelong number at $x_j$. Then $\psi = \sum_j \frac{1}{2^j} v_j \in \textup{PSH}(X,\theta)$ has positive Lelong numbers at all $x_j$. As we have argued in \cite[Theorem 1.1]{DDL16}, the Lelong numbers of $P_{\theta}[\psi]$ are the same as those of $\psi$, hence the model type potential $P_{\theta}[\psi]$ cannot have small unbounded locus.

The following convergence result is important in our later study, and it can be implicitly found in the arguments of \cite{BEGZ10}, as well as other works: 
\begin{lemma}
	\label{lem: basic convergence s.u.l}
	Let $u_k,u_k^j \in \textup{PSH}(X,\theta,\phi)$  and $C>0$ such that \[
	- C \leq u^j_k - \phi \leq C,  
	\] 
for all $j \in \Bbb N$ and $k \in \{1,\ldots,n\}$. Assume also that  $u_k^j \to u_k, k \in \{1,\ldots,n\}$ in capacity. Suppose also that  $f,{f_j}$ are uniformly bounded, quasi-continuous, such that $f_j \to f$ in  capacity. Then $f_j \theta_{u_1^j} \wedge ...\wedge \theta_{u_n^j} \rightarrow f \theta_{u_1} \wedge ...\wedge \theta_{u_n}$ weakly. 
\end{lemma}
\begin{proof}
Let $A \subset X$ be closed pluripolar such that $\{\phi = -\infty \} \subset A$.
We set $\mu_j:=\theta_{u_1^j} \wedge ...\wedge \theta_{u_n^j}$, and $\mu:=\theta_{u_1} \wedge ...\wedge \theta_{u_n}$.  Fix a continuous function $\chi$ on $X$, $\vep>0$ and $U$ an open relatively compact subset of $X\setminus A$ such that $\mu(X\setminus U) \leq \vep$. Fix $V$ a slightly larger open subset of $X\setminus A$ such that $U\Subset V\Subset X\setminus A$. Fix $\rho$ a continuous non negative function on $X$ which is supported in $V$ and is identically $1$ in $U$.   Since all functions $u_k^j$ are uniformly bounded in $V$ (along with $u_k$) it follows from \cite[Theorem 4.26]{GZ17} that $\chi f_j\mu_j$ converges weakly to $\chi f\mu$ in $V$. Also, Bedford-Taylor theory gives that $\mu_j$ converges weakly to $\mu$ in $V$. Thus $\liminf_{j} \mu_j(U) \geq \mu(U)$, hence $\limsup_j \mu_j(X\setminus U) \leq \mu(X\setminus U)\leq  \vep$ since $\mu_j(X)=\mu(X)$. Since $\chi, \rho, f_j,f$ are uniformly bounded it follows that $\limsup_j \int_{X\setminus U} \rho  |\chi f_j|\mu_j,  \limsup_j \int_{X\setminus U}    |\chi f_j|\mu_j,  \int_{X\setminus U}  \rho  |\chi f|\mu,  \int_{X\setminus U} |\chi f|\mu$ are all bounded by $ C\vep$  for some uniform constant $C>0$. On the other hand, since $\chi f_j \mu_j$ converges weakly to $\chi f \mu$ in $V$ and $\rho=0$ outside $V$, we have
	\[
	\lim_{j} \int_X \rho \chi f_j d\mu_j  = \int_X \rho \chi f d\mu.
	\]
    
    Thus,
    \begin{eqnarray*}
    \limsup_{j} \left|\int_X \chi f_j d\mu_j -\int_X \chi f d\mu\right| &\leq & \limsup_j \left|\int_X \rho \chi f_j d\mu_j -\int_X \rho \chi f d\mu\right| +4C\varepsilon
    \end{eqnarray*}
    It then follows that 
	\[
	\limsup_{j} \left|\int_X \chi f_j d\mu_j -\int_X \chi f d\mu\right| \leq C'\vep. 
	\]
	Letting $\vep \to 0$ we arrive at the conclusion. 
\end{proof}

\subsection{The relative Monge-Amp\`ere capacity}
We introduce the \emph{relative Monge-Amp\`ere capacity} of a Borel set $B \subset X$:
\begin{equation*}
\textup{Cap}_\phi(B):=\sup\left\{\int_B \theta^n_\psi, \ \psi\in \psh(X,\theta) ,\ \phi \leq \psi \leq \phi + 1\right\}.
\end{equation*}
Note that in the K\"ahler case  a related notion of capacity has been studied in \cite{DiLu14,DiLu15}. 
In the case when $\phi=V_{\theta}$ we recover the Monge-Amp\`ere capacity used in \cite[Section 4.1]{BEGZ10}. As is well known, the (generalized) Monge-Amp\`ere capacity and the global relative extremal functions play a vital role in establishing uniform estimates for  complex Monge-Amp\`ere equations (see \cite{Kol98}, \cite{BEGZ10}, \cite{DiLu14,DiLu15}). Along these lines the capacity  $\capa_{\phi}$ will play a crucial role in proving the regularity part of Theorem \ref{thm4}.

\begin{lemma}
	\label{lem: cap is inner regular}
	The relative Monge-Amp\`ere capacity $\capi$ is inner regular, i.e.
\[
\capi(E) =\sup \{\capi(K)\setdef K\subset E\ ; \ K \ \textrm{is compact}\}.
\]
\end{lemma}
\begin{proof}
By definition $\capi(E) \geq \capi(K) $ for any compact set $ K\subset E$. Fix $\vep>0$. There exists $u\in \psh(X,\theta)$ such that $\phi\leq u\leq \phi+1$ and 
\[
\int_E \theta_u^n \geq \capi(E)-\vep.
\]
Since $\theta_u^n$ is an inner regular Borel measure it follows that there exists a compact set $K\subset E$ such that $\int_K \theta_u^n \geq \int_E \theta_u^n -\vep\geq \capi(E)-2\vep$. Hence $\capi(K)\geq \capi(E)-2\vep$. Letting $\vep\to 0$  and taking the supremum over all the compact set $K\subset E$, we arrive at the conclusion. 
\end{proof}

By definition, $\textup{Cap}_\theta(B)\leq\textup{Cap}_\theta(X)=\int_X \theta_\phi^n$. Next we note that if $\textup{Cap}_\phi(B)=0$ then $B$ is a very ``small'' set:

\begin{lemma}\label{lem: zero_cap} Let $B \subset X$ be a Borel set. Then $\textup{Cap}_\phi(B)=0$ if and only if $B$ is pluripolar.
\end{lemma}
\begin{proof} Fix $\omega$ K\"ahler with $\omega \geq \theta$. Recall that a Borel subset $E\subset X$ is pluripolar if and only if $\capo(E)=0$ (see \cite[Corollary 3.11]{GZ05} which goes back to \cite{BT82}).  

If $B$ is pluripolar then $\capi(B)=0$ by definition. Conversely, assume that $\textup{Cap}_\phi(B)=0$. If $B$ is non-pluripolar then, $\capo(B)>0$. Since $\textup{Cap}_{\omega}$ is inner regular (\cite[Remark 1.7]{BBGZ13}), there exists  a compact subset $K$ of $B$  such that $\textup{Cap}_{\omega}(K)>0$. In particular $K$ is non-pluripolar, hence the global extremal function of $(K,\omega)$, $V_{\omega,K}^*$  is bounded from above (i.e. it is not identically $+\infty$) by \cite[Theorem 9.17]{GZ17}. Since $\omega\geq \theta$  we have $V_{\theta,K}^* \leq V_{\omega,K}^*$, hence $V_{\theta,K}^*$ is also bounded from above. 

We recall that $\theta_{V_{\theta,K}^*}^n$ is supported on $K$ (\cite[Theorem 9.17]{GZ17}), and we consider $u_t:=P_{\theta}(\phi+t,V_{\theta,K}^*), \ t>0$. By the argument of Corollary \ref{cor: non-collapsing} there exists $t_0>0$ big enough such that $\psi := u_{t_0} \in \textup{PSH}(X,\theta)$ has the same  singularity type as $\phi$ and $\int_K \theta^n_\psi >0$. We can assume that $\phi\leq \psi \leq \phi + C$ for some $C> 0$. If $C \leq 1$ then $\psi$ is a candidate in the definition of $\textup{Cap}_\phi(B)$, hence $\textup{Cap}_\phi(B) > 0$, which is a contradiction. In case $C>1$, then $(1-\frac{1}{C}) \phi + \frac{1}{C}\psi$ is a candidate in the definition of $\textup{Cap}_\phi(K)$, hence 
$$
\textup{Cap}_\phi(B) \geq \textup{Cap}_\phi(K)\geq \int_K \theta^n_{\left(1-\frac{1}{C}\right) \phi + \frac{1}{C}\psi} > \frac{1}{C^n} \int_K \theta_\psi^n >0,
$$ 
 a contradiction. 
\end{proof}
\subsubsection{The $\phi$-relative extremal function}
Recall that $\phi$ has small unbounded locus, i.e. $\phi$ is locally bounded outside a closed complete pluripolar subset $A\subset X$.  Recall that by ${\rm PSH}(X,\theta,\phi)$ we denote the set of all $\theta$-psh functions which are more singular than $\phi$. 

Let $E$ be a Borel subset of $X$. The relative extremal function of $(E,\phi,\theta)$ is defined as 
\[
h_{E,\phi}:= \sup\{u\in \psh(X,\theta,\phi) \setdef u\leq \phi-1\ \textrm{on}\ E\ ; \ u\leq 0\ \textrm{on}\ X\}.
\]
\begin{lemma}\label{basic property of relative extremal function phi}
Let $E$ be a Borel subset of $X$ and $h_{E,\phi}$ be the relative extremal function of $(E,\phi,\theta)$. Then $h_{E,\phi}^*$ is a $\theta$-psh function such that $\phi-1\leq h_{E,\phi}^*\leq \phi$. Moreover, $\theta_{h_{E,\phi}^*}^n$ vanishes on $\{h_{E,\phi}^*<0\}\setminus \bar{E}$. 
\end{lemma}
\begin{proof}
Since $\phi-1$ is a candidate defining $h_{E,\phi}$ it follows that $\phi-1\leq h_{E,\phi}\leq h_{E,\phi}^*$. Any $u \in \textup{PSH}(X,\theta,\phi)$ with $u \leq 0$ is a candidate of $P_\theta(\phi+C,0)$, for some $C\in \Bbb R$. By Theorem \ref{thm: ceiling coincide envelope non collapsing} we get that $u \leq P_{\theta}[\phi]=\phi$, hence $h_{E,\phi}^* \leq \phi$.

By the above, $h_{E,\phi}^*$ is locally bounded outside the closed pluripolar set $A$, and a standard balayage argument (see e.g. \cite{BT76}, \cite[Proposition 4.1]{GZ05}, \cite[Lemma 1.5]{BBGZ13})  gives that $\theta_{h_{E,\phi}^*}^n$ vanishes in $\{h_{E,\phi}^*<0\}\setminus \bar{E}$.
\end{proof}

\begin{theorem}\label{thm: capi formula compact}
	If  $K$ is a compact subset of $X$ and $h:=h_{K,\phi}^*$ then 
	\[
	\capi(K) =\int_K \theta_{h}^n =\int_X (\phi-h) \theta_h^n. 
	\]
\end{theorem}
\begin{proof}
Set $h:=h^*_{K,\phi}$ and observe that $h+1$ is a candidate defining $\textup{Cap}_{\phi}$.  Since $\theta_h^n$ puts no mass on the set $\{h<\phi\}\setminus K$ and $h=\phi-1$ on $K$ modulo a pluripolar set we thus get 
\[
\capi(K) \geq \int_K \theta_h^n  = \int_X (\phi-h)\theta_h^n. 
\]

Now let $u$ be a $\theta$-psh function such that $\phi-1\leq u\leq \phi$. For a fixed $\vep\in (0,1)$ set $u_{\vep}:=(1-\vep)u +\vep \phi$. Since $h=\phi-1$ on $K$ modulo a pluripolar set and $\phi-1\leq u_{\varepsilon}$ it follows that  $K\subset \{h<u_{\vep}\}$ modulo a pluripolar set. By the comparison principle we then get 
\[
(1-\vep)^n\int_K \theta_u^n \leq \int_{\{h<u_{\vep}\}} \theta_{u_{\vep}}^n \leq  \int_{\{h<u_{\vep}\}}  \theta_h^n =\int_ K \theta_h^n,
\]
where in the last equality we use the fact that $\theta_h^n$ vanishes in $\{h<0\}\setminus K$. Since $u$ was taken arbitrarily, letting $\vep\to 0$ we obtain $\capi(K)\leq \int_K \theta_h^n$. This together with the previous step gives the result.  
\end{proof}

\begin{coro}
	\label{cor: ext capacity of compact sets}
	If $(K_j)$ is a decreasing sequence of compact sets then 
	\[
	\capi(K) =\lim_{j\to +\infty} \capi(K_j),
	\]
    where $K:= \bigcap_j K_j$.
	In particular, for any compact set $K$ we have 
	\[
	\capi(K) =\inf\{\capi(U) \setdef K\subset U\subset X\ ; \ U \ \textrm{is open in}\ X\}.
	\]
\end{coro}

\begin{proof}Let $h_j:=h^*_{K_j,\phi}$ be the relative extremal function of $(K_j,\phi)$. Then  $(h_j)$ increases almost everywhere to $h\in \psh(X,\theta)$ which satisfies $\phi-1\leq h\leq \phi$, since $\phi-1\leq h_j\leq \phi$.  

Next we claim that  $\theta_h^n(\{h<0\}\setminus K)=0$. Indeed, for $m\in \mathbb{N}$ fixed and for each $j>m$ we have that $\{h<0\}\setminus K_m \subset \{h_j<0\}\setminus K_j$ and by Lemma \ref{basic property of relative extremal function phi},
$$\theta_{h_j}^n (\{h_j<0\}\setminus K_j) =0.
$$
Using the continuity of the Monge-Amp\`ere measure along monotone sequences (Theorem \ref{thm: lsc of MA measures} and Remark \ref{rem: increasing implies capacity}) we have that $\theta_{h_j}^n$ converges weakly to $\theta_h^n$. Since $\{h<0\}\setminus K_m$ is open it follows that 
 \[
 \theta_{h}^n (\{h<0\}\setminus K_m)\leq \liminf_{j\to +\infty} \theta_{h_j}^n(\{h<0\}\setminus K_m)=0.
 \]
 The claim follows as $m\rightarrow +\infty$. 
 It then  follows from Theorem \ref{thm: capi formula compact}  and Lemma \ref{lem: basic convergence s.u.l}  that 
\begin{flalign*}
	\lim_{j\to +\infty}\capi(K_j) =  \lim_{j\to +\infty} \int_X (\phi-h_j) \theta_{h_j}^n = \int_X (\phi-h)\theta_h^n=\int_K \theta_h^n \leq \capi(K).  
\end{flalign*}
As the reverse inequality is trivial, the first statement follows.	

To prove the last statement, let $(K_j)$ be a decreasing sequence of compact sets such that $K$ is contained in the interior of $K_j$ for all $j$. Then by the first part of the corollary we have that 
	\begin{eqnarray*}
	\capi(K) =\lim_{j\to +\infty} \capi(K_j)&\geq & \lim_{j\to +\infty} \capi(\textrm{Int}(K_j)) \\
    &\geq & \inf\{\capi(U) \setdef K\subset U\subset X\ ; \ U \ \textrm{is open in}\ X\},
	\end{eqnarray*}
    hence equality.
\end{proof}
\begin{coro}
	If $U$ is an open subset of $X$ then 
	\[
	\capi(U) = \int_X (\phi-h_{U,\phi})\theta_{h_{U,\phi}}^n. 
	\]
\end{coro}

\begin{proof}
	Let $(K_j)$ be an increasing sequence of compact subsets of $U$ such that $\cup K_j=U$. For each $j$ we set $h_j:=h^*_{K_j,\phi}$. By Theorem \ref{thm: capi formula compact} we have that 
	\[
	\capi(K_j) = \int_X (\phi-h_j)\theta_{h_j}^n. 
	\]
	Since $h_j$ decreases to $h_{U,\phi}$ it follows from Lemma \ref{lem: basic convergence s.u.l} that the right-hand side above converges to $\int_X (\phi-h_{U,\phi})\theta_{h_{U,\phi}}^n$. 
 Moreover, by the argument of Lemma \ref{lem: cap is inner regular} we have $\lim_j \capi(K_j) = \capi (U)$, hence the result follows.
\end{proof}

\subsubsection{The global $\phi$-extremal function}
For a Borel set $E \subset X$, we define the  global $\phi$-{\it extremal function} of $(E,\phi,\theta)$ by  

$$
V_{E,\phi}:=\sup\left\{\psi \in \textup{PSH}(X,\theta,\phi), \ \psi\leq \phi
\textup{ on } E \right \}.$$
We then introduce the \emph{relative Alexander-Taylor capacity} of $E$,
$$
T_\phi(E):=\exp(-M_{\phi}(E)), \ \ \text{where}\ M_{\phi}(E) : =\sup_X V_{E,\phi}^*. 
$$ 
Paralleling Lemma \ref{lem: zero_cap}, we have the following result:
\begin{lemma}\label{lem: Alexander Taylor capacity} Let $E \subset X$ be a Borel set. If $M_\phi (E)=+\infty$, then $E$ is pluripolar. 
\end{lemma}
\begin{proof}
Let $\omega$ be a K\"ahler form such that $\omega \geq \theta$.  By definition we have 
\[
V_{E, \phi} \leq V_{E, \omega}:=\sup\left\{\psi \in \textup{PSH}(X,\omega), \ \psi\leq 0 \textup{ on } E \right \}.
\]
 This clearly implies $M_\phi (E)\leq \sup_X V_{E, \omega}^* $, and so by assumption we know that $\sup_X V_{E, \omega}^*=+\infty $. It then follows from \cite[Theorem 5.2]{GZ05} that $E$ is pluripolar. 
\end{proof}
If $M_\phi (E)<+\infty$ then $V^*_{E, \phi}\in \textup{PSH}(X, \theta)$, and standard  arguments give that $\theta_{V_{E,\phi}^*}^n$ does not charge $X \setminus \overline{E}$ (see \cite[Theorem 9.17]{GZ17} or \cite[Theorem 5.2]{GZ05}).
Now, we claim that
\begin{equation} \phi \leq V_{E,\phi}^* \leq P_{\theta}[\phi] + M_\phi(E)=\phi + M_\phi(E).\label{eq: V_K_est}
\end{equation}
The first inequality simply follows by definition, since $\phi \leq 0$ is a candidate in the definition of $V_{E,\phi}$. If $M_\phi (E)=+\infty$ then the second inequality holds trivially. Assume that $M_\phi (E)<+\infty$. The inequality then holds, since $V_{E,\phi}^* -M_\phi(E)\leq 0$, and each candidate potential $\psi$ in the definition of  $V_{E,\phi}^*$ is more singular than $\phi$, i.e., $\psi-M_\phi(E)$ is a candidate in the definition of $P_\theta (\phi+C, 0)$, for some $C>0$. Finally, the last identity follows from Theorem \ref{thm: ceiling coincide envelope non collapsing}.  

In particular, since $\phi$ has small unbounded locus, so does the usc regularization $V_{E,\phi}^*$. Also, from \eqref{eq: V_K_est} we deduce that  if  $M_\phi(E)<+\infty$, the $\theta$-psh functions $V_{E,\phi}^*$ and $\phi$ have the same singularity type, hence Proposition \ref{prop: comparison generalization} insures that $$\int_X \theta_{V_{E,\phi}^*}^n = \int_ X \theta_\phi^n.$$
The Alexander-Taylor and Monge-Amp\`ere capacities are related by the following estimates: 

\begin{lemma}\label{lem: compcap} Suppose $K \subset X$ is a compact subset and $\textup{Cap}_\phi(K)>0$. Then we have
$$1\leq\bigg(\frac{\int_X \theta_\phi^n}{\textup{Cap}_\phi(K)}\bigg)^{1/n}\leq \max(1,M_\phi(K)).$$
\end{lemma} 
\begin{proof}
The first  inequality is trivial. We now prove the second inequality. Note that we can assume that $M_\phi(K)<+\infty$, since otherwise the inequality is trivially satisfied. We then consider two cases. If $M_\phi(K)\leq 1$, then $V_{K,\phi}^* \leq \phi +1$, hence $V_{K,\phi}^*$ is a candidate in the definition $\textup{Cap}_\phi(K)$. Since $\theta_{V_{K,\theta}^*}^n$ is supported on $K$, we thus have 
\[
\textup{Cap}_\phi(K)\ge\int_K \theta_{V_{K,\phi}^*}^n=\int_X\theta_{V_{K,\phi}^*}^n=\int_X \theta_\phi^n,
\]
and the desired inequality holds in this case.

If $M:=M_\phi(K)\ge 1$, then by \eqref{eq: V_K_est} we have $\phi \le M^{-1}V_{K,\phi}^*+(1-M^{-1})\phi \le \phi+1,$ and by definition of the relative capacity we can write:
$$\textup{Cap}_\phi(K)\ge\int_K\ \theta_{M^{-1}V_{K,\phi}^*+(1-M^{-1})\phi}^n \geq  \frac{1}{M^n} \int_K \theta_{V^*_{K,\phi}}^n =  \frac{1}{M^n} \int_X \theta_{V^*_{K,\phi}}^n=\frac{1}{M^n}\int_X \theta_\phi^n,$$
implying the desired inequality.
\end{proof}

\subsection{The relative finite energy class $\mathcal E^1(X,\theta,\phi)$}

To develop the variational approach to \eqref{eq: CMAE_sing}, we need to understand the relative version of  the Monge-Amp\`ere energy, and its bounded locus $\mathcal E^1(X,\theta,\phi)$. For $u\in \Ec(X,\theta, \phi) $ with relatively minimal singularities, we define the Monge-Amp\`ere energy of $u$ relative to $\phi$ as 
\[
\mathrm{I}_{\phi} (u) :=\frac{1}{n+1} \sum_{k=0}^n \int_X (u-\phi) \theta_u^k \wedge\theta_{\phi}^{n-k}. 
\]

In the next theorem we collect basic properties of the Monge-Amp\`ere energy:

\begin{theorem}\label{thm: basic I energy} Suppose $u,v \in \mathcal E(X,\theta,\phi)$ have relatively minimal singularities. The following hold:\\
\noindent (i) $ \mathrm{I}_{\phi}(u)-\mathrm{I}_{\phi}(v) = \frac{1}{n+1}\sum_{k=0}^n \int_X (u-v) \theta_{u}^k \wedge \theta_{v}^{n-k}.$\\
\noindent (ii) If $u\leq \phi$ then, $
\int_X (u-\phi) \theta_u^n \leq I_{\phi}(u) \leq \frac{1}{n+1} \int_X (u-\phi) \theta_{u}^n. $ \\
\noindent (iii) $\mathrm{I}_{\phi}$ is non-decreasing and concave along affine curves. Additionally, the following estimates hold: $
	\int_X (u-v) \theta_u^n \leq I_{\phi}(u) -I_{\phi}(v) \leq \int_X (u-v) \theta_v^n.$
\end{theorem}
\begin{proof}
Since $\phi$ has small unbounded locus, it is possible to repeat  the arguments of  \cite[Proposition 2.8]{BEGZ10} almost word for word. As a courtesy to the reader  the detailed proof is presented here.

To start, we note that the  non-pluripolar products  appearing in our arguments are simply the mixed  Monge-Amp\`ere measures defined  in the sense of Bedford and Taylor \cite{BT76} on $X\setminus A$, where $A$ is a closed complete pluripolar subset of $X$, such that $\phi$ is locally bounded on $X \setminus A$ (consequently, $u$ and $v$ are locally bounded in on $X \setminus A$). Since $u-v$ is globally bounded on $X$,  we can perform integration by parts in our arguments below, via \cite[Theorem 1.14]{BEGZ10}. 

For any fixed $k\in \{0,...,n-1\}$, set $T=\theta_u^k \wedge\theta_{v}^{n-k-1}$. Using integration by parts \cite[Theorem 1.14]{BEGZ10}, we can write
\begin{flalign}
\nonumber\int_X (u-v) \theta_u^k \wedge\theta_{v}^{n-k}& = \int_X (u-v)(\theta +i\ddbar v) \wedge T  \\
\nonumber &=\int_X  (u-v) i\ddbar (v-u)\wedge T +  \int_X (u-v) i\ddbar u \wedge T+\int_X (u-v)\theta \wedge T\\
\nonumber &= \int_X  (v-u)   i\ddbar (u-v)\wedge T + \int_X (u-v)  \theta_u \wedge T \\
&\geq  \int_X (u-v) \theta_u\wedge T=\int_X (u-v) \theta_u^{k+1} \wedge\theta_{v}^{n-k-1} ,\label{eq: basic_est}
\end{flalign}
where in the last inequality we used that $\int_X (-\varphi)i\ddbar \varphi \wedge T=i \int_X \partial \varphi \wedge \bar \partial \varphi \wedge T\geq 0$ with $\varphi:=u-v$. This shows in particular that the sequence $k\mapsto \int_X (u-\phi) \theta_{u}^k \wedge \theta_{\phi}^{n-k}$ is non-increasing in $k$, verifying (ii).

Now we compute the derivative of  $f(t):=I_{\phi}(u_t), t\in [0,1]$, where $u_t:=tu +(1-t)v$. By the multi-linearity property of the non-pluripolar product we see that $f(t)$ is a polynomial in $t$. Using again integration by parts \cite[Theorem 1.14]{BEGZ10}, one can check the following formula:
\begin{eqnarray*}
f'(t) &=&\frac{1}{n+1}\bigg(\sum_{k=0}^n \int_X (u-v) \theta_{u_t}^k\wedge \theta_{\phi}^{n-k} +\sum_{k=1}^n \int_X k(u_t-\phi) i\ddbar (u-v)\wedge \theta_{u_t}^{k-1}\wedge \theta_{\phi}^{n-k}\bigg)\\
&=& \frac{1}{n+1}\bigg(\sum_{k=0}^n \int_X (u-v) \theta_{u_t}^k\wedge \theta_{\phi}^{n-k} +\sum_{k=1}^n \int_X k(u-v) (\theta_{u_t}-\theta_{\phi})\wedge \theta_{u_t}^{k-1}\wedge \theta_{\phi}^{n-k}\bigg)\\
&=& \int_X (u-v) \theta_{u_t}^n. 
\end{eqnarray*}
Computing one more derivative, we arrive at
$$f''(t)=n \int_X (u-v) i\ddbar (u-v) \wedge \theta_{u_t}^{n-1}= -n i \int_X \partial (u-v) \wedge \bar \partial (u-v) \theta_{u_t}^{n-1} \leq 0.$$
This shows that $\AMO$ is concave along affine curves. 

Now, the function $t\mapsto f'(t)$ is continuous on $[0,1]$, thanks to convergence property of the Monge-Amp\`ere operator (see Lemma \ref{lem: basic convergence s.u.l}). It thus follows that 
\[
I_{\phi}(u_1) -I_{\phi}(u_0) = \int_0^1 f'(t)dt = \int_0^1 \int_X (u-v) \theta_{u_t}^n dt. 
\]
Using the multi-linearity of the non-pluripolar product again, we get that
\begin{eqnarray*}
\int_0^1 \int_X (u-v) \theta_{u_t}^n dt &=& \sum_{k=0}^n \left(\int_0^1\binom{n}{k} t^k (1-t)^{n-k} dt\right) \int_X   (u-v) \theta_{u}^k \wedge \theta_v^{n-k}\\
&=&  \frac{1}{n+1}\sum_{k=0}^n \int_X   (u-v) \theta_{u}^k \wedge \theta_v^{n-k}.
\end{eqnarray*}
This verifies (i), and another application of \eqref{eq: basic_est} finishes the proof of (iii).
\end{proof}

\begin{lemma}\label{lem: convergence of AM} Suppose $u_j,u \in \Ec(X,\theta,\phi)$ have relatively minimal singularities such that $u_j$ decreases to $u$. Then $\AM_{\phi}(u_j)$ decreases to $\mathrm{I}_{\phi}(u)$.   
\end{lemma}
\begin{proof} From Theorem \ref{thm: basic I energy}(iii) it follows that $|I_\phi(u_j) - I_\phi(u)|=I_\phi(u_j) - I_\phi(u) \leq \int_X (u_j - u)\theta_u^n$. An application of the dominated convergence theorem finishes the argument.
\end{proof}

We can now define the Monge-Amp\`ere energy for arbitrary $u\in \mathrm{PSH}(X,\theta,\phi)$ using  a familiar formula:
\[
I_{\phi}(u) : =\inf \{I_{\phi}(v) \setdef  v\in \Ec(X,\theta,\phi), \; v\ \textrm{has relatively minimal singularities, and } u\leq v\}. 
\]
\begin{lemma}\label{lem: easy convergence AMO}
	If $u\in \psh(X,\theta,\phi)$ then $\AMO(u)=\lim_{t\to \infty} \AMO(\max(u,\phi-t))$. 
\end{lemma}
\begin{proof}
	It follows from the above definition that $\AMO(u)\leq \lim_{t\to \infty} \AMO(\max(u,\phi-t))$. Assume now that $v\in \psh(X,\theta,\phi)$ is such that $u\leq v$, and $v$ has the same singularity type as $\phi$ (i.e. $v$ is a candidate in the definition of $I_{\phi}(u)$). Then for $t$ large enough we have $\max(u,\phi-t)\leq v$, hence the other inequality follows from monotonicity of $\AMO$.
\end{proof}

We let $\mathcal{E}^1(X,\theta,\phi)$ denote the set of all $u\in \psh(X,\theta,\phi)$  such that $\AMO(u)$ is finite. As a result of Lemma \ref{lem: easy convergence AMO} and Theorem \ref{thm: basic I energy}(iii) we observe that $\AMO$ is non-decreasing in $\psh(X,\theta,\phi)$. Consequently, $\mathcal{E}^1(X,\theta,\phi)$ is stable under the max operation, moreover we have the following familiar characterization of $\mathcal E^1(X,\theta,\phi)$:

\begin{lemma}\label{charcterization class E^1}
Let $u\in \mathrm{PSH}(X,\theta,\phi)$. Then $u\in \mathcal{E}^1(X,\theta,\phi)$ if and only if $u\in \mathcal{E}(X,\theta,\phi)$ and $\int_X (u-\phi) \theta_u^n >-\infty$. 
\end{lemma}
\begin{proof} We can assume that $u\leq \phi$. For each $C>0$ we set $u^C:= \max(u, \phi-C)$. If $I_{\phi}(u) >-\infty$ then by the monotonicity property we have $I_\phi(u^C) \geq  I_{\phi}(u)$. Since $u^C \leq \phi$, an application of Theorem \ref{thm: basic I energy}(ii) gives that $\int_X (u^C-\phi) \theta_{u^C}^n\geq -A, \ \forall C$, for some  $A>0$. From this we  obtain that 
\[
\int_{\{u\leq \phi-C\}} \theta_{u^C}^n \leq \frac{A}{C} \to 0, 
\]
as $C\to +\infty$. Hence it follows from Lemma \ref{lem: E_def_lem} that $u\in \mathcal{E}(X,\theta,\phi)$. Moreover by the plurifine property of the non-pluripolar product  we have that
\[
\int_X (u^C -\phi) \theta_{u^C}^n \leq \int_{\{u>\phi-C\}} (u-\phi)\theta_u^n.
\]
Letting $C\to \infty$ we see that $\int_X (u-\phi)\theta_u^n >-A$. 

To prove the reverse statement,  assume that $u\in \mathcal{E}(X,\theta,\phi)$ and $\int_X (u-\phi)\theta_u^n >-\infty$.
For each $C>0$ since $\theta_u^n$ and $\theta_{u^C}^n$ have the same mass and coincide in $\{u>\phi-C\}$ it follows that $\int_{\{u\leq \phi-C\}} \theta_{u^C}^n =\int_{\{u\leq \phi-C\}} \theta_{u}^n$.  From this we deduce that 
\begin{eqnarray*}
\int_X (u^C-\phi)  \theta_{u^C}^n &=&-\int_{\{u\leq \phi-C\}} C \theta_u^n  + \int_{\{u>\phi-C\}}(u-\phi) \theta_{u}^n = \int_X (u-\phi) \theta_u^n>-A. 
\end{eqnarray*}
It thus follows from Theorem \ref{thm: basic I energy}(ii) that $I_{\phi}(u^C)$ is uniformly bounded.  Finally, it follows from Lemma \ref{lem: easy convergence AMO}  that $I_{\phi}(u^C)\searrow I_{\phi}(u)$ as $C\to \infty$, finishing the proof. 
\end{proof}

We finish this subsection with a series of small results listing various properties of the class $\mathcal E^1(X,\theta,\phi)$:

\begin{lemma}\label{lem: AM continuous along decreasing sequence}
Assume that $(u_j)$ is a sequence in $\mathcal{E}^1(X,\theta,\phi)$ decreasing to $u\in \mathcal{E}^1(X,\theta,\phi)$. Then $\AMO(u_j)$ decreases to $\AMO(u)$.
\end{lemma}
\begin{proof}
Without loss of generality we can assume that $u_j\leq \phi$ for all $j$. For each $C>0$ we set $u_j^C:=\max (u_j, \phi-C)$ and $u^C:= \max(u, \phi-C)$. Note that $u_j^C, u^C$ have the same singularities as $\phi$. Then Lemma \ref{lem: convergence of AM} insures that $\lim_{j} I_{\phi}(u_j^C) = I_{\phi}(u^C)$. Monotonicity of $I_\phi$ gives now that $I_{\phi}(u) \leq \lim_j I_\phi(u_j) \leq \lim_j I_\phi(u_j^C)= I_\phi(u^C)$. Letting $C \to \infty$, the result follows.
\end{proof}

\begin{lemma}
	\label{lem: criteria in E1}
	Assume that $(u_j)$ is a decreasing sequence in $\Ec^1(X,\theta,\phi)$  such that $\AMO(u_j)$ is uniformly bounded. Then the limit $u:=\lim_{j} u_j$ belongs to $\Ec^1(X,\theta,\phi)$ and $\AMO(u_j)$ decreases to $\AMO(u)$.
\end{lemma}
\begin{proof}
We can assume that $u_j\leq \phi$ for all $j$. Since $\AMO(u_j)\leq \int_X (u_j-\phi) \theta_{\phi}^n$,  $\AMO(u_j)$ is uniformly bounded and $\theta_{\phi}^n$ has bounded density with respect to $\omega^n$, it follows that $\int_X u_j \omega^n$ is uniformly bounded, hence $u \neq -\infty$. 

By continuity along decreasing sequences (Lemma \ref{lem: AM continuous along decreasing sequence}) we have $\lim_{j\rightarrow +\infty} \AMO(\max(u_j, \phi-C))= \AMO(\max(u, \phi-C))$)  .  It follows that  $\AMO(\max(u,\phi-C))$ is uniformly bounded. Lemma \ref{lem: easy convergence AMO} then insures that  $\AMO(u)$ is finite, i.e., $u\in \Ec^1(X,\theta,\phi)$. 
\end{proof}

\begin{coro}\label{cor: I_conc_on_E}
$\AMO$ is concave along affine curves in $\psh(X,\theta,\phi)$. In particular, the set $\Ec^1(X,\theta,\phi)$ is convex. 
\end{coro}
\begin{proof}
Let $u,v\in \psh(X,\theta,\phi)$ and $u_t:=tu+(1-t)v, t\in (0,1)$. If one of $u,v$ is not in $\Ec^1(X,\theta,\phi)$ then the conclusion is obvious. So, we can assume that both $u$ and $v$ belong to $\Ec^1(X,\theta,\phi)$. For each $C>0$ we set $u_t^C:=t \max(u,\phi-C) + (1-t)\max(v,\phi-C)$. By Theorem \ref{thm: basic I energy})(iii), $t \to \AMO(u_t^C)$ is concave. Since $u_t^C$ decreases to $u_t$ as $C\to \infty$, Lemma  \ref{lem: criteria in E1} gives the conclusion. 
\end{proof}

\subsection{The variational method}
Recall that $\phi$ is a $\theta$-psh function with small unbounded locus such that $\phi=P_\theta[\phi]$,  and $\int_X \theta_{\phi}>0$. For this subsection we additionally  normalize our class so that $\int_X \theta_{\phi}^n=1$. 

We adapt the variational method of \cite{BBGZ13} to solve the complex Monge-Amp\`ere equations in our more general setting: 

\begin{equation}
\label{eq: MA lambda}
\theta_u^n = e^{\lambda u} \mu, \ u\in \Ec(X,\theta,\phi). 
\end{equation}
where $\lambda\geq 0$, $\mu$ is a positive non-pluripolar measure on $X$. If $\lambda=0$ then  we also assume that $\mu(X)=1$ which is a necessary condition for the equation to be solvable.  

We introduce the following functionals on $\mathcal{E}^1(X,\theta,\phi)$: 
\[
F_{\lambda}(u) :=F_{\lambda,\mu}(u):= \AMO(u) -L_{\lambda,\mu}(u) , \ u\in \Ec^1(X,\theta,\phi),
\]
where $L_{\lambda,\mu}(u) :=\frac{1}{\lambda}\int_X e^{\lambda u} d\mu$ if $\lambda>0$ and $L_{\mu}(u):=L_{0,\mu}(u):=\int_X (u-\phi)d\mu$. 
Note that when $\lambda>0$, $F_{\lambda}$ is finite on $\Ec^1(X,\theta,\phi)$. It is no longer the case if $\lambda=0$ in which case  we will restrict ourself to the following set of measures. For each constant $A\geq 1$ we let $\mathcal{M}_A$ denote the set of all probability measures $\mu$ on $X$ such that 
\begin{equation*}
\mu(E)\leq A \cdot \capi(E), \ \textrm{for all Borel subsets} \ E\subset X.
\end{equation*}
\begin{lemma}
	\label{lem: MA compact convex}
	$\mathcal{M}_A$ is a compact convex subset of the set of probability measures on $X$. 
\end{lemma}
\begin{proof}
	The convexity is obvious. We now prove that $\mathcal{M}_A$ is closed. Assume that $(\mu_j)\subset \mathcal{M}_A$ is a sequence converging weakly to a probability measure $\mu$. Then for any open set $U$ we have 
	\[
	\mu(U) \leq \liminf_{j}\mu_j(U) \leq A\capi(U).
	\]
Now, let $K\subset X$ be a compact subset. Taking the infimum over all open sets $U\supset K$ in the above inequality, it follows from Corollary \ref{cor: ext capacity of compact sets} that $\mu(K)\leq A \capi(K)$.  Since $\mu$  and $\capi$ are inner regular (Lemma \ref{lem: cap is inner regular}) it follows that the inequality holds for all Borel sets, finishing the proof.
	\end{proof}

\begin{lemma}
	\label{lem: first condition of M_A}
If $\mu\in \mathcal{M}_A$ then $F_{0,\mu}$ is finite on $\Ec^1(X,\theta,\phi)$. Moreover, there is a constant $B>0$ depending on $A$ such that for all $u\in \psh(X,\theta,\phi)$ with $\sup_X u=0$ we have 
	\[
	\int_X (u-\phi)^2 d\mu \leq B (|\AMO(u)|+1). 
	\]
\end{lemma}
The proof given below is inspired by \cite[Lemma 2.9]{BBGZ13}.
\begin{proof}
	Fix  $u\in \psh(X,\theta,\phi)$ such that    $\sup_X u=0$. By considering $u_k:=\max(u,\phi-k)$ and then letting $k\to +\infty$, we can assume that $u-\phi$ is bounded.  We first prove that 
    \begin{equation}
    	\label{eq: proof of Lemma 4.18}
    	 \int_{1}^{+\infty} t\capi(u<\phi-2t) dt \leq C(-\AMO(u) +1),
    \end{equation}
    for some uniform constant $C:=C(n)>0$.
    
    Indeed, for each $t>1$ we set $u_t:=t^{-1}u+(1-t^{-1})\phi$.  We also fix $\psi\in \psh(X,\theta)$ such that $\phi-1\leq \psi\leq \phi$. Observe that $u_t, \psi \in \mathcal{E}( X,\theta, \phi)$) and that the following inclusions hold
	\[
	(u<\phi-2t) \subset (u_t <\psi-1) \subset (u<\phi-t), \ t>1. 
	\]
 It thus follows that 
	\begin{equation}\label{ineq measures}
	\theta_{\psi}^n(u<\phi-2t) \leq \theta_{\psi}^n(u_t<\psi-1) \leq \theta_{u_t}^n (u_t<\psi-1)\leq \theta_{u_t}^n (u<\phi-t),
	\end{equation}
   where in the second inequality we used the comparison principle (see Corollary \ref{comparison principle}).
	Expanding $\theta_{u_t}^n$ we see that 
	\begin{equation}\label{ineq u}
	\theta_{u_t}^n \leq Ct^{-1}\sum_{k=1}^n \theta_{u}^k\wedge \theta_{\phi}^{n-k} +  \theta_{\phi}^n, \ \ \forall t>1, 
	\end{equation}
for a uniform constant $C=C(n)$. Since $\theta_{\phi}^n$ has bounded density with respect to Lebesgue measure (see Theorem \ref{thm: MA measure of WN envelope}), using \cite[Theorem 2.50]{GZ17} we infer that 
\begin{equation}
	\label{eq: GZ 05 CLN} 
	\theta_{\phi}^n (u<\phi-t) \leq A \int_{\{u\leq - t\}} \omega^n \leq    A e^{-at}, 
\end{equation}
for some uniform constants $a,A>0$ depending only on $n,\omega,X$. 
  Combining \eqref{eq: GZ 05 CLN} with \eqref{ineq measures} and \eqref{ineq u} we get that 
	\begin{eqnarray*}
		\int_1^{\infty} t\theta_{\psi}^n(u<\phi-2t)dt &\leq &  \int_1^{\infty} t\theta_{u_t}^n(u<\phi-t)dt\\
        &\leq &  C \int_1^{\infty}  \sum_{k=0}^n \theta_{u}^k\wedge \theta_{\phi}^{n-k}(u<\phi-t)dt + \int_1^{\infty}  t\theta_{\phi}^n(u<\phi-t) dt\\
		&\leq &C(n+1) |\AMO(u)| + C'.
	\end{eqnarray*}
Taking the supremum over all candidates $\psi+	1$ we arrive at
    \[
    \int_1^{+\infty} t\capi(u<\phi-2t) dt \leq C(n+1) |\AMO(u)| + C',
    \]
    proving \eqref{eq: proof of Lemma 4.18}.
    Finally, we can write
	\begin{eqnarray*}
		\int_X (u-\phi)^2 d\mu &=& 2 \int_0^{+\infty} t\mu(u<\phi-t)dt\leq 4 + 8\int_1^{+\infty} t\mu(u<\phi-2t)dt\\
        &\leq & 4 +8 \int_1^{+\infty} At\capi(u<\phi-2t)dt \leq B(|I_{\phi}(u)|+1),
	\end{eqnarray*}
	where $B>0$ is a uniform constant depending on $n,C,C'$.  
\end{proof}
Observe that Lemma \ref{lem: first condition of M_A} above together with H\"older inequality give that $F_{0, \mu}$ is finite on $\Ec^1(X, \theta, \phi)$ whenever $\mu\in \mathcal{M}_A$ for some $A\geq 1$. Indeed
\begin{equation}\label{F_0 finite}
\int_X |u-\phi| d\mu \leq \left(\int_X (u-\phi)^2 d\mu \right)^{1/2} \mu(X)^{1/2}\leq  C(|I_{\phi}(u)|^{1/2}+1)
\end{equation}
for a suitable $C>0$.
\subsubsection{Maximizers are solutions}

\begin{prop}\label{prop: usc of AMO}
$\AMO:\Ec^1(X,\theta,\phi) \to \Bbb R$  is upper semicontinuous with respect to the weak $L^1$ topology of potentials. 
\end{prop}
\begin{proof}
	Assume that $(u_j)$ is a sequence in $\mathcal{E}^1(X,\theta,\phi)$ converging in $L^1$ to $u\in \Ec^1(X,\theta,\phi)$. We can assume that $u_j\leq 0$ for all $j$.  For each $k,\ell \in \mathbb{N}$ we set $v_{k,\ell}:= \max(u_{k},...,u_{k+\ell})$. As $\Ec^1(X,\theta,\phi)$ is stable under the max operation,  we have that $v_{k,\ell}\in \Ec^1(X,\theta,\phi)$. 

Moreover $v_{k,\ell}\nearrow \varphi_k:= \left(\sup_{j \geq k} u_j\right)^{*}$, hence by the monotonicity property we get $\AMO(\varphi_k)\geq \AMO(v_{k,\ell})\geq \AMO(u_k)>-\infty$. As a result, $\varphi_k\in  \Ec^1(X,\theta,\phi)$.
     By Hartogs' lemma $\varphi_k \searrow u$ as $k\to \infty$.  By Lemma \ref{lem: AM continuous along decreasing sequence} it follows that $\AMO(\varphi_k)$ decreases to $\AMO(u)$. Thus, using the monotonicity of $\AMO$ we get $\AMO(u)= \lim_{k\rightarrow \infty} \AMO(\varphi_k) \geq \limsup_{k\rightarrow \infty} \AMO(u_k),$ finishing the proof.
\end{proof}

Next we describe the first order variation of $I_\phi$, shadowing a result from \cite{BB10}:

\begin{prop}\label{prop: derivative of AMO}
Let $u\in \mathcal{E}^1(X,\theta,\phi)$ and $\chi$ be a continuous function on $X$. For each $t>0$ set $u_t:= P_{\theta}(u+t\chi)$. Then $u_t\in \mathcal{E}^1(X,\theta,\phi)$, $t\mapsto \AMO(u_t)$ is differentiable, and its derivative is given by
\[
\frac{d}{dt} \AMO(u_t) = \int_X \chi \theta_{u_t}^n, \ t \in \Bbb R. 
\]
\end{prop}
\begin{proof}
Note that $u+t \inf_X \chi $ is a candidate in each envelope, hence $u+t \inf_X \chi\leq u_t$. Monotonicity of $I_\phi$ now implies that $u_t \in \mathcal E^1(X,\theta,\phi)$.

As the singularity type of each $u_t$ is the same, we can apply Lemma \ref{lem: inequalities energy} below and conclude: 
\[
 \int_X (u_{t+s}-u_t) \theta_{u_{t+s}}^n \leq \AMO(u_{t+s})-\AMO(u_t) \leq   \int_X (u_{t+s}-u_t) \theta_{u_{t}}^n.
\]
It follows from \cite[Proposition 2.13]{DDL16} that $\theta_{u_t}^n$ is supported on $\{u_t=u+t\chi\}$. We thus have
\[
\int_X (u_{t+s}-u_t) \theta_{u_{t}}^n = \int_X (u_{t+s}-u-t\chi) \theta_{u_t}^n\leq \int_X s\chi \theta_{u_t}^n,
\]
since $u_{t+s}\leq u+(t+s)\chi$. Similarly we have 
\[
\int_X (u_{t+s}-u_t) \theta_{u_{t+s}}^n = \int_X (u+(t+s)\chi-u_t) \theta_{u_{t+s}}^n\geq \int_X s\chi \theta_{u_{t+s}}^n.
\]
Since $u_{t+s}$ converges uniformly to $u_t$  as $s\to 0$, by Theorem \ref{thm: lsc of MA measures} it follows that $\theta_{u_{t+s}}^n$ converges weakly to $\theta_{u_t}^n$. As $\chi$ is continuous, dividing by $s>0$ and letting $s\to 0^+$ we see that the right derivative of $\AMO(u_t)$ at $t$ is $\int_X \chi \theta_{u_t}^n$. The same argument applies for the left derivative.
\end{proof}

\begin{lemma}\label{lem: inequalities energy}
Suppose $u,v \in \mathcal E^1(X,\theta,\phi)$ have the same singularity type. Then 
$$ \int_X (u-v) \theta_{u}^n \leq \AMO(u)-\AMO(v) \leq   \int_X (u-v) \theta_{v}^n.
$$
\end{lemma}
\begin{proof}
First, note that these estimates hold for $u^C := \max(u,\phi -C), v^C:=\max(v,\phi -C)$, by Theorem \ref{thm: basic I energy}(iii). It is easy to see that $u^C - v^C$ is uniformly bounded and converges to $u - v$. Also, by the comments after Lemma \ref{lem: E_def_lem} it follows that the measures $\theta_{v^C}^n$ converge uniformly to $\theta_{v}^n$ (not just weakly!). Putting these last two facts together, the dominated convergence theorem gives that
\begin{flalign*}
\bigg|\int_X (u^C - v^C) &\theta_{v^C}^n-\int_X (u - v) \theta_{v}^n\bigg| \leq \\
&\leq \bigg|\int_X (u^C - v^C) (\theta_{v^C}^n-\theta_{v}^n)\bigg| + \bigg|\int_X (u^C - v^C) \theta_{v}^n-\int_X (u - v) \theta_{v}^n\bigg| \to 0,
\end{flalign*}
as $C \to \infty$.  A similar convergence statement holds for the left hand side of our double estimate as well, and using Lemma \ref{lem: easy convergence AMO}, the result follows.
\end{proof}

\begin{theorem}
\label{thm: maximizers are solutions}
Assume that $L_{\lambda,\mu}$ is finite on $\Ec^1(X,\theta,\phi)$ and  $u\in \Ec^1(X,\theta,\phi)$ maximizes $F_{\lambda,\mu}$ on $\mathcal{E}^1(X,\theta,\phi)$. Then $u$ solves the equation \eqref{eq: MA lambda}. 
\end{theorem}

\begin{proof}
First, let's assume that $\lambda \neq 0$. Let $\chi$ be an arbitrary continuous function on $X$ and set $u_t:= P_{\theta}(u+t\chi)$. It follows from  Proposition \ref{prop: derivative of AMO} that $u_t\in \Ec^1(X,\theta,\phi)$ for all $t\in \mathbb{R}$, that the function 
$$g(t):=\AMO(u_t) -L_{\lambda,\mu}(u+t\chi)$$ 
is differentiable on $\mathbb{R}$, and its derivative is given by $g'(t)=\int_X \chi \theta_{u_t}^n-\int_X \chi e^{\lambda (u+t\chi)}d\mu$. Moreover, as $u_t\leq u+t\chi$, we have $g(t) \leq F_{\lambda, \mu}(u_t)\leq \sup_{\Ec^1(X,\theta,\phi)} F_{\lambda ,\mu}=F(u)=g(0)$. This means that $g$ attains a maximum at $0$, hence  $g'(0)=0$. Since $\chi$ was taken arbitrary  it follows that $\theta_u^n=e^{\lambda u}\mu$. When $\lambda =0$, similar arguments give the conclusion.
\end{proof}

\subsubsection{The case $\lambda>0$}

Having computed the first order variation of the Monge-Amp\`ere energy, we establish the following  existence and uniqueness result. 

\begin{theorem}\label{thm: existence_MA_eq_exp}
Assume that $\mu$ is a positive non-pluripolar measure on $X$ and $\lambda>0$. Then there exists a unique $\varphi\in \Ec^1(X,\theta,\phi)$ such that 
\begin{equation}\label{eq: MA_exp_version}
	\theta_{\varphi}^n =e^{\lambda \varphi} \mu. 
\end{equation}
\end{theorem}
\begin{proof}
	We use the variational method as above (see also \cite{DDL16}). It suffices to treat the case $\lambda=1$ as the other cases can de done similarly. Consider
	\[
	F(u) := \AMO(u)-\int_X e^{u} d\mu, \ u\in \Ec^1(X,\theta,\phi). 
	\] 
	Let $(\varphi_j)$ be a sequence in $\mathcal{E}^1(X,\theta,\phi)$ such that $\lim_{j}F(\varphi_j)=\sup_{\mathcal{E}^1(X,\theta,\phi)}F>-\infty$. We claim that $\sup_X \varphi_j$ is uniformly bounded from above. Indeed, assume that it were not the case. Then by relabeling the sequence we can assume that $\sup_X \varphi_j$ increase to $+\infty$. By compactness property \cite[Proposition 2.7]{GZ05} it follows that the sequence $\psi_j:=\varphi_j-\sup_X \varphi_j$ converges in $L^1(X,\omega^n)$ to some $\psi\in \psh(X,\theta)$ such that $\sup_X \psi=0$.  In particular $\int_X e^{\psi} d\mu >0$. It thus follows that 
\begin{equation}
	\label{eq: proof of Theorem 2.23 1}
	\int_X e^{\varphi_j}d\mu = e^{\sup_X \varphi_j}  \int_X e^{\psi_j} d\mu  \geq c e^{\sup_X \varphi_j}
\end{equation}
for some positive constant $c$. Note also that $\psi_j \leq \phi$ since $\psi_j\in \Ec(X,\theta,\phi)$ and $\psi_j\leq 0$ and $\phi$ is the maximal function with these properties (see Theorem \ref{thm: ceiling coincide envelope non collapsing}). It then follows that 
\begin{equation}
	\label{eq: proof of Theorem 2.23 2}
\AMO(\varphi_j) = \AMO(\psi_j) + \sup_X \varphi_j \leq \sup_X \varphi_j. 
\end{equation}
From \eqref{eq: proof of Theorem 2.23 1} and \eqref{eq: proof of Theorem 2.23 2} we arrive at 
\[
\lim_{j\to +\infty} F(\varphi_j) \leq  \lim_{j\to +\infty} (\sup_X \varphi_j - ce^{\sup_X \varphi_j}) = -\infty, 
\]
which is a contradiction. Thus $\sup_X \varphi_j$ is bounded from above as claimed. 
 Since $F(\varphi_j)\leq\AMO(\varphi_j)\leq \sup_X \varphi_j$ it follows that $\AMO(\varphi_j)$ and hence $\sup_X \varphi_j$ is also bounded from below. It follows again from \cite[Proposition 2.7]{GZ05} that a subsequence of $\varphi_j$ (still denoted by $\varphi_j$) converges in $L^1(X,\omega^n)$ to some $\varphi\in \psh(X,\theta)$. Since $\AMO$ is upper semicontinuous  it follows that $\varphi\in \Ec^1(X,\theta,\phi)$. Moreover, by continuity of $u\mapsto \int_Xe^{u}d\mu$ we get that $F(\varphi) \geq \sup_{\Ec^1(X,\theta,\phi)} F$.  Hence $\varphi$ maximizes $F$ on $\Ec^1(X,\theta,\phi)$. Now  Theorem \ref{thm: maximizers are solutions} shows that $\varphi$ solves the desired complex Monge-Amp\`ere equation. The next lemma address the uniqueness question.
\end{proof}
\begin{lemma}
	\label{lem: comp_sol_subsol}
	Let $\lambda>0$. Assume that $\varphi\in \Ec(X,\theta,\phi)$ is a solution of \eqref{eq: MA_exp_version} while $\psi\in \Ec(X,\theta,\phi)$ satisfies $\theta_{\psi}^n \geq e^{\lambda \psi}\mu.$
	Then $\varphi\geq \psi$ on $X$. 
\end{lemma} 
\begin{proof}
	By the comparison principle for the class $\Ec(X,\theta,\phi)$ (Corollary \ref{comparison principle}) we have 
	$$
	\int_{\{\varphi<\psi\}} \theta_{\psi}^n \leq \int_{\{\varphi<\psi\}} \theta_{\varphi}^n.
	$$
	As $\varphi$ is a solution and $\psi$ is a subsolution to \eqref{eq: MA_exp_version} we also have
	$$
	\int_{\{\varphi<\psi\}} e^{\lambda \psi} d\mu \leq \int_{\{\varphi<\psi\}} \theta_{\psi}^n \leq \int_{\{\varphi<\psi\}} \theta_{\varphi}^n = \int_{\{\varphi<\psi\}} e^{\lambda \varphi} d\mu\leq \int_{\{\varphi<\psi\}} e^{\lambda \psi} d\mu.
	$$
	It follows that all inequalities above are equalities, hence $\varphi\geq \psi$ $\mu$-almost everywhere on $X$. Since $\mu=e^{-\lambda \varphi} \theta_\varphi^n$, it follows that $\theta_\varphi^n (\{\varphi<\psi\})=0$. By the domination principle \ref{prop: domination principle} we get that $\varphi\geq \psi$ everywhere on $X$.
\end{proof}

\subsubsection{The case $\lambda=0$}

\begin{theorem}
	\label{thm: existence in M_A}
	Assume that $\mu\in \mathcal{M}_A$ for some $A\geq 1$. Then there exists $u\in \Ec^1(X,\theta,\phi)$ such that $\theta_{u}^n=\mu$. 
\end{theorem}
\begin{proof}
	In view of Theorem \ref{thm: maximizers are solutions} it suffices to find a maximizer in $\Ec^1(X,\theta,\phi)$ of the functional $F:=F_{0,\mu}$ defined by 
	\[
	F(u) := \AMO(u) -\int_X (u-\phi)d\mu, \ u\in \Ec^1(X,\theta,\phi). 
	\] 
	Note that $F(u)$ is finite for all $u\in \Ec^1(X,\theta,\phi)$ since $\mu\in \mathcal{M}_A$ (see Lemma \ref{lem: first condition of M_A}). 
	Let $(u_j)$ be a sequence in $\Ec^1(X,\theta,\phi)$ such that $\sup_X u_j=0$ and $F(u_j)$ increase to $\sup_{\Ec^1(X,\theta,\phi)} F>-\infty$. By compactness property \cite{GZ05} a subsequence of $(u_j)$ converges to $u\in \psh(X,\theta,\phi)$, and $\sup_X u=0$. Moreover, since $\mu\in \mathcal{M}_A$, by \eqref{F_0 finite} we have that 
	\[
	F(u_j) \leq \AMO(u_j) +C|\AMO(u_j)|^{1/2}+C, \ \forall j. 
	\]
	It thus follows that $\AMO(u_j)$ is uniformly bounded. Since $\AMO$ is upper semicontinuous it follows that $u\in \Ec^1(X,\theta,\phi)$. Also, since $\int_X (u_j-\phi)^2d\mu$ is uniformly bounded (Lemma \ref{lem: first condition of M_A}) it follows from the same arguments of \cite[Lemma 11.5]{GZ17} that $\int_X (u_j-\phi)d\mu$ converge to $\int_X (u-\phi)d\mu$. Since $\AMO$ is usc we obtain that $F(u)\geq \limsup_{j}F(u_j)$. Hence $u$ maximizes $F$ on $\mathcal{E}^1(X,\theta,\phi)$, and the result follows. 
\end{proof}

\begin{lemma}
\label{lem: decomposition}
If $\mu$ is a positive non-pluripolar measure on $X$ and $A\geq 1$ then there exists $\nu\in \mathcal{M}_A$ and $0\leq f\in L^1(X,\nu)$ such that $\mu=f\nu$.
\end{lemma}
The short proof given below is due to Cegrell \cite{Ce98}.
\begin{proof}
It follows from Lemma \ref{lem: MA compact convex} that $\mathcal{M}_A$ is a convex compact subset of $\mathcal{M}(X)$, the space of probability measures on $X$. It follows from \cite[Lemma 1]{KS} that we can write 
\[
\mu = \nu + \sigma,
\]
where $\nu,\sigma$ are non-negative Borel measures on $X$ such that $\nu$ is absolutely continuous with respect to an element in $\mathcal M_A$ and $\sigma$ is singular with respect to any element of $\mathcal M_A$, i.e. $\sigma \perp m$ for any $m\in \mathcal M_A$. It then follows from \cite[Theorem]{Rai69} that $\sigma$ is supported on a Borel set $E$ such that $m(E)=0$ for all $m\in \mathcal M_A$. If $u$ is a candidate defining the capacity $\capi(E)$, then clearly $\theta_u^n\in \mathcal M_A$, hence $\int_E \theta_u^n=0$. It follows that $\capi(E)=0$, hence by Lemma \ref{lem: zero_cap} $E$ is pluripolar. Therefore, $\sigma=0$ since $\mu$ does not charge pluripolar sets. 
\end{proof}

To prove the main existence result in this subsection we also need the following  lemma. The argument uses the locality of non-pluripolar Monge-Amp\`ere measures with respect to the plurifine topology, and 
is identical with the proof of \cite[Corollary 1.10]{GZ07}. 

\begin{lemma}\label{lem: max principle}
Assume that $\nu$ is a positive non-pluripolar Borel measure on $X$ and $u,v\in \psh(X,\theta)$. If $\theta_u^n \geq \nu$ and $\theta_v^n \geq \nu$ then $\theta_{\max(u,v)}^n\geq \nu$. 
\end{lemma}

\begin{theorem}\label{thm: existence lambda=0}
Assume that $\mu$ is a positive non-pluripolar measure on $X$ such that $\mu(X)=\int_X \theta_{\phi}^n$. Then there exists  $u\in \Ec(X,\theta,\phi)$ (unique up to a constant) such that $\theta_{u}^n=\mu$. 
\end{theorem}
\begin{proof}
It follows from Lemma \ref{lem: decomposition} that $\mu=f\nu$ where $\nu\in \mathcal{M}_1$ and $0\leq f \in L^1(X,\nu)$. 
For each $j$ it follows from Theorem \ref{thm: existence in M_A} that there exists $u_j\in \Ec^1(X,\theta,\phi)$ such that $\sup_X u_j=0$ and 
\[
\theta_{u_j}^n = c_j\min(f,j) \nu. 
\]
Here, $c_j$ is a normalization constant and $c_j\to 1$ as $j\to +\infty$. We can assume that $1\leq c_j\leq 2$ for all $j$. By compactness \cite[Proposition 8.5]{GZ17} a subsequence of $(u_j)$ converges in $L^1(X,\omega^n)$ to $u\in \psh(X,\theta,\phi)$ with $\sup_X u=0$. We will show that $u\in \Ec(X,\theta,\phi)$. For each $k\in \NN$ we set $v_k:=(\sup_{j\geq k} u_j)^*$. Then $v_k\in \Ec^1(X,\theta,\phi)$ and $(v_k)$ decreases pointwise to $u$. For each $k$ fixed, and for all $j>k$ we have $\theta_{u_j}^n \geq \min(f,k)\nu$.  Thus for all $\ell \in \mathbb{N}$ it follows from Lemma \ref{lem: max principle} that $\theta_{w_{k,\ell}}^n \geq \min(f,k)\nu$, where $w_{k,\ell}:= \max(u_k,\cdots, u_{k+\ell})$. Since $(w_{k,\ell})$ increases almost everywhere to $v_k$ as $\ell\to +\infty$ it follows from Theorem \ref{thm: lsc of MA measures} and Remark \ref{rem: increasing implies capacity} that 
\[
\theta_{v_k}^n \geq \min(f,k) \nu. 
\]
Thus for each $C>0$ setting $v_k^C:=\max(v_k,V_{\theta}-C)$, using the plurifine property of the  Monge-Amp\`ere measure and observing that $\{u>V_{\theta}-C\} \subseteq \{v_k>V_\theta-C\}$ we have 
\[
\theta_{v_k^C}^n \geq \mathbbm{1}_{\{v_k>V_\theta-C\}} \theta_{v_k}^n \geq\mathbbm{1}_{\{v_k>V_\theta-C\}} \min(f,k)\nu \geq \mathbbm{1}_{\{u>V_{\theta}-C\}} \min(f,k)\nu.
\] 
Since $(v_k^C)$ decreases to $u^C:=\max(u,V_{\theta}-C)$ and $v_k^C, u^C \in \Ec(X, \theta)$, it follows  from Theorem \ref{thm: lsc of MA measures}  that $\theta_{v_k^C}^n$ converges weakly to $\theta_{u^C}^n$, hence 
\[
\theta_{u^C}^n \geq \mathbbm{1}_{\{u>V_{\theta}-C\}} \mu. 
\]
Since $\mu$ is non-pluripolar it follows by letting $C\to +\infty$ that 
\[
\theta_u^n=\lim_{C\to +\infty} \mathbbm{1}_{\{u>V_{\theta}-C\}} \theta_{u^C}^n \geq \lim_{C\to +\infty}\mathbbm{1}_{\{u>V_{\theta}-C\}} \mu= \mu.
\]
Moreover by  \cite[Theorem 1.2]{WN17} the total mass of $\theta_u^n$ is smaller than $\int_X \theta_{\phi}^n=\mu(X)$ since $u\leq \phi$. Hence $\int_X \theta_{\phi}^n=\mu(X)=\int_X \theta_u^n$. It thus follows that $u\in \Ec(X,\theta, \phi)$ and $\theta_u^n=\mu$. Uniqueness is addressed in the next theorem.
\end{proof}

\begin{theorem}
\label{thm: uniqueness}
Assume $u,v\in \mathcal{E}(X,\theta,\phi)$ are such that $\theta_u^n=\theta_v^n$. Then $u-v$ is constant.
\end{theorem}
The proof of this uniqueness result rests on the adaptation of the mass concentration technique of Ko{\l}odziej and Dinew  \cite{Dw09b} to our more general setting (see also \cite{BEGZ10}, \cite{DL15}). The arguments carry over almost verbatim, but as a courtesy to  the reader we provide a detailed account. 
\begin{proof}
Set $\mu:=\theta_u^n=\theta_v^n$. We will prove that there exists a constant $C$ such that $\mu$ is supported on $\{u=v+C\}$. This will allow to apply the domination principle (Proposition \ref{prop: domination principle}) to insure the conclusion. Assume that it is not the case. Arguing exactly as in \cite[Section 3.3]{BEGZ10} we can assume that $0<\mu(U)<\mu(X)=\int_X \theta_{\phi}^n$ and $\mu(\{u=v\})=0$, where $U:=\{u<v\}$. Let $c>1$ be a normalization constant such that $\int_{\{u<v\}} c^n d\mu =\mu(X)$. It follows from Theorem \ref{thm: existence lambda=0} that there exists $h\in \Ec(X,\theta,\phi)$, $\sup_X h=0$, such that $\theta_h^n =c^n\mathbbm{1}_U \mu$. In particular, $h\leq \phi$. For each $t\in (0,1)$ we set $U_t:=\{(1-t)u+t\phi<(1-t)v+th\}$ and note that, since $h\leq \phi$, the sets $U_{t}$ increase as $t\to 0^+$ to $U\setminus \{h=-\infty\}$. 

By the mixed Monge-Amp\`ere inequalities  \cite[Proposition 1.11]{BEGZ10} (which go back to Dinew \cite{Dw09a} and Ko{\l}odziej \cite{Kol03}), we have that 
\begin{equation}\label{ineq1}
\theta_u^{n-1}\wedge \theta_h \geq \mathbbm{1}_U c\mu, \ \ \ \theta_u^k\wedge\theta_v^{n-k}\geq \mu, \ \ k=0,...,n. 
\end{equation}
Moreover, since $u,v,h \in \mathcal{E}(X, \theta, \phi)$, it follows from Corollary \ref{cor: convexity} that all the above non-pluripolar products have the same mass. Consequently, $\theta_u^k\wedge\theta_v^{n-k}= \mu, \ k=0,...,n$. 
Using the partial comparison principle (Proposition \ref{prop: general CP}) we can write that 
\[
\int_{U_t} \theta_u^{n-1}\wedge \theta_{(1-t)v+th} \leq \int_{U_t} \theta_{u}^{n-1}\wedge \theta_{(1-t)u+t\phi}.
\]
Expanding, and using the fact that $\theta_u^n=\theta_u^{n-1}\wedge \theta_v$ we get 
\begin{equation}\label{ineq2}
\int_{U_t} \theta_u^{n-1}\wedge \theta_{h} \leq \int_{U_t} \theta_{u}^{n-1}\wedge \theta_{\phi}.
\end{equation}
Combining \eqref{ineq1} and \eqref{ineq2} we have
$
c\mu(U_t)\leq  \int_{U_t} \theta_u^{n-1} \wedge \theta_h \leq \int_{U_t} \theta_u^{n-1}\wedge \theta_{\phi}.$ 
Letting $t\to 0$, and noting that $\mu$ is non-pluripolar (hence $\mu$ put no mass on the set $\{h=-\infty\}$)  we obtain 
\[
c\mu(U) \leq \int_{U} \theta_u^{n-1}\wedge \theta_{\phi}.
\]
Now, applying the same arguments for $V:=\{u>v\}$ we obtain 
\[
b\mu(V) \leq \int_{V} \theta_u^{n-1}\wedge \theta_{\phi},
\]
where $b>1$ is a  constant so that $b^n\mu(V)=\mu(X)$. Using that  $\mu(\{u=v\})=0$, we can sum up the last two inequalities and  obtain
\[
0<\min(b,c) \mu(X)\leq \int_X \theta_u^{n-1}\wedge\theta_{\phi} =\mu(X).
\]
where the last equality follows again from Corollary \ref{cor: convexity}. This is a contradiction since $\min(b,c)>1$. 
\end{proof}

\subsection{Regularity of solutions}

Recall that we work with $\phi \in \psh(X,\theta)$ with small unbounded locus such that $P_{\theta}[\phi]=\phi$, and $\int_X \theta_\phi^n >0$. Let $f \in L^p(\omega^n)$ with $f \geq 0$. In the previous subsection we have shown that the equation
\begin{equation*}
\theta_\psi^n = f \omega^n, \ \ \psi \in \mathcal E^1(X,\theta,\phi)
\end{equation*}
has a unique solution. In this subsection we will show that this solution has the same singularity type as $\phi$. This generalizes \cite[Theorem B]{BEGZ10}, that treats the particular case of solutions with minimal singularities in a big class. Analogous results will be obtained for the solutions of \eqref{eq: MA_exp_version} as well.

Our arguments will closely follow the path laid out in \cite[Section 4.1]{BEGZ10}, which builds on fundamental work of Ko{\l}odziej in the K\"ahler case (see \cite{Kol98,Kol03}). As we shall see, the fact that $\phi$ has model type singularity plays a vital role in making sure that the methods of \cite{BEGZ10} work in our more general context as well.

We first prove that any measure with $L^{1+\varepsilon}$, $\varepsilon>0$, density is dominated by the relative capacity: 

\begin{prop}\label{prop: capLp} 
Let $f \in L^p(\omega^n), \ p >1$ with $f \geq 0$.
Then there exists $C>0$ depending only on $\theta,\omega,p$ and $\|f\|_{L^p}$ such that 
$$\int_E f \omega^n \le \frac{C}{\big(\int_X \theta_\phi^n\big)^2} \cdot \textup{Cap}_\phi (E)^2$$
for all  Borel sets $E  \subset X$. 
\end{prop}

\begin{proof}

Since $\capi$ is inner regular we can  assume that $E$ is compact. Thanks to Lemma \ref{lem: Alexander Taylor capacity} we can also assume that $M_\phi(E)< +\infty$. 

We introduce $\nu_\theta:=\sup_{T,x}\nu(T,x)$, where $x\in X$, $T$ is any closed positive $(1,1)$-current cohomologous with $\theta$, and $\nu(T,x)$ denotes the Lelong number of $T$ at $x$. 
As a result, the uniform version of Skoda's integrability theorem  \cite[Theorem 2.50]{GZ17}  yields a constant $C>0$, only depending on $\theta$ and $\omega$ such that $\int_X\exp(-\nu_\theta^{-1}\psi)\omega^n\le C$
for all $\psi \in \textup{PSH}(X,\theta)$ with $\sup_X\psi=0$. Applying this to $V_{E,\phi}^*-M_\phi(E)$ we get
$$\int_X\exp(-\nu_\theta^{-1}V_{E,\phi}^*) \omega^n \le C \cdot \exp(-\nu_{\theta}^{-1} M_\phi(E)).$$

On the other hand, $ V_{E,\phi}^*\le 0 $ on $E$ a.e.~with respect to Lebesgue measure, hence
\begin{equation}\label{equ:vol} \vol_\omega(E):=\int_E \omega^n\le C \cdot \exp(-\nu_\theta^{-1}M_\phi(E)).
\end{equation}
An application of H\"older's inequality gives  
\begin{equation}\label{equ:holder}
\int_E f \omega^n \le\|f\|_{L^{p}}\vol_\omega(E)^{\frac{p-1}{p}}.
\end{equation}
At this point we may assume that $M_{\phi}(E)\geq 1$. Indeed, if this were not the case, then Lemma \ref{lem: compcap} would imply that $\textup{Cap}_\phi(E)=\int_X \theta_\phi^n$, yielding the desired estimate of the proposition. Putting together Lemma~\ref{lem: compcap}, (\ref{equ:vol}) and (\ref{equ:holder}) we get:
$$\int_E f \omega^n \le  C^{p-1/p}\cdot \| f\|_{L^{p}} \cdot \exp\left(-\frac{p-1}{p\nu_\theta}\bigg(\frac{\textup{Cap}_\phi(E)}{\int_X \theta_\phi^n}\bigg)^{-1/n}\right).$$ 
The result now follows, as 
$\exp(-t^{-1/n})=O(t^2)$ when $t\to 0_+$. 
\end{proof}

Before we state the main result of this subsection, we need one last lemma, which is a simple consequence of our comparison principle:

\begin{lemma} \label{lem: majoCap}
Let $u \in \mathcal E(X,\theta,\phi)$. Then for all $t>0$ and $\delta \in (0,1]$ we have
$$\textup{Cap}_\phi \{u<\phi-t-\delta\} \leq \frac{1}{\delta^{n}}\int_{\{u<\phi-t\}} \theta_u^n.
$$
\end{lemma}
\begin{proof}
Let $\psi\in \textup{PSH}(X,\theta,\phi)$ be such that $\phi\le\psi \le \phi+1$. In particular, note that $\psi \in \mathcal E(X,\theta,\phi) $. We then have
$$\{u<\phi-t-\delta\}\subset\{u<\delta\psi+(1-\delta)\phi-t-\delta\}\subset\{u<\phi-t\}.$$
Since $\delta^n\theta_\psi^n\le \theta^n_{\delta\psi+(1-\delta)\phi}$, $u$ has relative full mass and $\mathcal E(X,\theta,\phi)$ is convex, Corollary \ref{comparison principle} yields
\begin{eqnarray*}
\delta^n\int_{\{u<\phi-t-\delta\}} \theta^n_\psi &\le & \int_{\{u<\delta\psi+(1-\delta)\phi-t-\delta\}}\theta^n_{\delta\psi+(1-\delta)\phi}\\
&\le & \int_{\{u<\delta \psi+(1-\delta)\phi-t-\delta\}}\theta^n_{u}\le\int_{\{u<\phi-t\}}\theta_u^n.
\end{eqnarray*}
Since $\psi$ is an arbitrary candidate in the definition of $\textup{Cap}_\phi$, the proof is complete.
\end{proof}
We arrive at the main results of this subsection:
\begin{theorem}\label{thm: min sing of solution}
Suppose $\phi =P_\theta[\phi]$ has small unbounded locus and $\int_X \theta_\phi^n >0$. Let also $\psi \in \mathcal E(X,\theta,\phi)$ with $\sup_X \psi =0$. If $\theta_\psi^n = f \omega^n$ for some $f \in L^p(\omega^n), \ p > 1,$ then $\psi$ has the same singularity type as $\phi$, more precisely:
$$\phi - C\Big(\|f\|_{L^p},p,\omega,\theta,\int_X \theta_\phi^n\Big)  \leq \psi \leq \phi.$$
\end{theorem} 

\begin{proof} To begin, we introduce the function
$$g(t):=\big(\textup{Cap}_\phi\{\psi<\phi-t\}\big)^{1/n}, \ t \geq 0.$$ 
We will show that $g(M)=0$ for some $M$ under control. By Lemma \ref{lem: zero_cap} we will then have $\psi \geq \phi-M$ a.e. with respect to $\omega^n$, which then implies $\psi \geq \phi-M$ on $X$.

Since $\theta_\psi^n=f \omega^n$, it follows from Proposition \ref{prop: capLp} and Lemma~\ref{lem: majoCap} that
$$
g(t+\delta) \leq \frac{C^{1/n}}{\delta} g(t)^2, \ \  t>0, \ 0<\delta<1.
$$
Consequently, we can apply \cite[Lemma 2.3]{EGZ09} to conclude that  
$g(M)=0$ for $M:=t_0+2$. As an important detail, the constant $t_0>0$ has to be chosen so that 
$$g(t_0)<\frac{1}{2C^{1/n}}.$$ 
On the other hand, Lemma \ref{lem: majoCap} (with $\delta=1$) implies that
$$
g(t+1)^n\le \int_{\{\psi<\phi-t-1\}} f\omega^n \le\frac{1}{t+1}\int_X|\phi-\psi|f \omega^n
\leq \frac{1}{t+1} \| f\|_{L^{p}}(\|\psi\|_{L^{q}}+\| \phi\|_{L^{q}}),
$$
where in the last estimate we have used H{\"o}lder's inequality with $q = p/(p-1)$. Since $\psi$ and $\phi$ both belong to the compact set of $\theta$-psh functions normalized by $\sup_X u=0$, their $L^{q}$ norms are bounded by an absolute constant only depending on $\theta$, $\omega$ and $p$. Consequently, it is possible to choose $t_0$ to be only dependent on $\|f \|_{L^p},\theta, \omega,\int_X \theta_\phi^n$ and $p$, finishing the proof.
\end{proof}

\begin{coro}
	Suppose $\phi =P_\theta[\phi]$ has small unbounded locus and $\int_X \theta_\phi^n >0$. If $\lambda>0$ and, $\psi \in \Ec(X,\theta,\phi)$, $\theta_\psi^n = e^{\lambda \psi}f \omega^n$ for some $f \in L^p(\omega^n), \ p > 1,$ then $\psi$ has the same singularity type as $\phi$.
\end{coro} 

\begin{proof} 
Since $\psi$ is bounded from above on $X$ and $\lambda>0$ it follows that $e^{\lambda \psi} f\in L^p(X,\omega^n)$, $p>1$. The result follows from Theorem \ref{thm: min sing of solution}. 
\end{proof}

\subsection{Naturality of model type singularities and examples}

Our readers may still wonder if our choice of model potentials is a natural one in the discussion of  complex Monge-Amp\`ere equations with prescribed singularity. We hope to address the doubts in the next result.

\begin{theorem}\label{thm: naturality_of_model} Suppose $\psi \in \textup{PSH}(X,\theta)$ has small unbounded locus and the equation 
$$\theta_u^n = f \omega^n$$
has a solution $u \in \textup{PSH}(X,\theta)$ with the same singularity type as $\psi$, for all $f \in L^{\infty}, \ f \geq 0$ satisfying $\int_X \theta_\psi^n = \int_X f \omega^n > 0$. Then $\psi$ has model type singularity. 
\end{theorem}

\begin{proof}Our simple proof follows the guidelines of the example described in the beginning of Section \ref{sec 4}. Indeed, suppose that $[\psi]$ is not of model type. Then $P_{\theta}[\psi]$ is strictly less singular than $\psi$, 
but of course $\mathcal E(X,\theta,\psi) \subset \mathcal E(X,\theta,P_{\theta}[\psi])$, as $\int_X \theta_\psi^n = \int_X \theta_{P_{\theta}[\psi]}^n$. 

By Theorem \ref{thm: MA measure of WN envelope}, there exists $g \in L^\infty$ such that $\theta^n_{P_{\theta}[\psi]} = g \omega^n.$ 
By the uniqueness theorem (Theorem \ref{thm: uniqueness}), $P_{\theta}[\psi]$ is the only solution of this last equation inside $\mathcal E(X,\theta,P_{\theta}[\psi])$. 

Since $\mathcal E(X,\theta,\psi) \subset \mathcal E(X,\theta,P_{\theta}[\psi])$, but $P_{\theta}[\psi] \notin \mathcal E(X,\theta,\psi)$, we get that $\theta_u^n = g \omega^n$ cannot have any solution that has the same singularity type as $\psi$.
\end{proof}

Next we point out a simple way to construct model singularity types:

\begin{prop}\label{prop: Lp example} Suppose that $\psi \in \textup{PSH}(X,\theta)$ has small unbounded locus and $\theta_\psi^n = f \omega^n$ for some $f \in L^p(\omega^n), \ p > 1$ with $\int_X f \omega^n >0$. Then $\psi$ has model type singularity.
\end{prop}
\begin{proof}
We first observe that $\psi \in \Ec(X,\theta,P_{\theta}[\psi])$. Since $\theta_{\psi}^n$ has $L^p$ density with $p>1$, it thus follows from Theorem  \ref{thm: min sing of solution} that  $\psi-P_{\theta}[\psi]$ is bounded on $X$, hence $[\psi]=[P_\theta[\psi]]$, implying that $\psi$ has model type singularity.
\end{proof}

Using this simple proposition, one can show that all analytic singularity types are of model type, which was previously known to be true using algebraic methods (see \cite{RWN,RS}):

\begin{prop}\label{prop: analytic example} Suppose $\psi \in \textup{PSH}(X,\theta)$ has analytic singularity type, i.e. $\psi$ can be locally written as $c \log \big(\sum_j |f_j|^2\big) + g$, where $f_j$ are holomorphic, $c > 0$ and $g$ is smooth. Then $[\psi]$ is of model type.
\end{prop}
\begin{proof} We can assume that our fixed K\"ahler form $\omega$ satisfies $\omega\geq 2\theta$. Since $P_{\theta}[\psi]\leq P_{\omega}[\psi]$ it suffices to prove that $\psi-P_{\omega}[\psi]$ is globally  bounded on $X$. In fact we will prove the following stronger result: 
	\begin{equation}
		\label{eq: Lp density analytic sing}
	\rho := \frac{\omega_{\psi}^n}{\omega^n} \in L^p(\omega^n) , \  \textrm{ for some }\ p>1. 
	\end{equation}
	As  $\omega/2 \geq \theta$ it follows that $\int_X \omega_\psi^n  \geq 2^{-n}\int_X \omega^n>0$,  hence  Proposition \ref{prop: Lp example} will imply that $\psi-P_{\omega}[\psi]$ is globally  bounded on $X$.
	
	We now prove \eqref{eq: Lp density analytic sing}.  Since $X$ is compact it suffices to prove that  there exists a small open neighborhood  $U$ around a given point $x\in X$ (which will be fixed)  such that $\rho \in L^p(U,dV)$ for some $p>1$.   Since $\psi$ has analytic singularities we can find a holomorphic coordinate chart $\Omega$ around $x$ such that
	\[
	\psi = c\log \sum_{j=1}^N |f_j|^2 + g
	\] 
in a neighborhood of $\Omega$, where $c>0$ is a constant, $f_j$ are holomorphic functions in $\Omega$ and $g$ is a smooth  real-valued function in $\Omega$.  Let $A>0$ be large enough so that $(A-1)\omega +i\ddbar g\geq 0$ in $\Omega$. 
	
	In $X\setminus \{\psi=-\infty\}$, since $\psi$ is smooth we can write $\omega_{\psi}^n =\rho \omega^n$, where $\rho\geq 0$ is smooth. We extend $\rho$ to be $0$ over the set $\{\psi=-\infty\}$.  Then $\rho \omega^n$ is the non-pluripolar Monge-Amp\`ere measure of $\psi$ with respect to $\omega$ as follows from \cite{BEGZ10}, hence
	 \[
	 \int_{\Omega} \rho \omega^n \leq \int_X \rho \omega^n\leq  \int_X \omega^n.
	 \]
	  Similarly we can write  $ (A\omega+ i \ddbar  \psi)^n  =\rho_A \omega^n$ in $\Omega\setminus \{\psi=-\infty\}$, where $0\leq \rho_A \in L^1(\Omega,dV)$. 
	
	Now, we carry out the computation in $\Omega\setminus \{\psi=-\infty\}$. For notational convenience we set $h:= \sum_{j=1}^N |f_j|^2$, $\varphi:= \log \sum_{j=1}^N |f_j|^2$ and we  compute $i\ddbar \varphi$:
    \[
    i\ddbar \varphi =  \frac{\sum_{j=1}^N i\partial f_j \wedge \overline{\partial f_j }}{h }  -   \frac{i\left ( \sum_{j=1}^N \overline{f_j}\partial f_j\right )\wedge \left ( \sum_{j=1}^N {f_j}\overline {\partial f_j} \right ) }{h^2}.
    \]
    For each $1\leq j<k\leq N$ we set $\alpha_{j,k}:= f_j \partial f_k -f_k\partial f_j$. Then we obtain
    \begin{equation}\label{computation analytic}
     i\ddbar \varphi = h^{-2} \sum_{j < k} i \alpha_{j,k} \wedge \overline{\alpha_{j,k}}.
    \end{equation}

    Let $C>0$ be large enough such that $C^{-1}\beta \leq A\omega+i\ddbar g \leq C \beta$ in $\Omega$, where $\beta$ is the standard K\"ahler form in $\mathbb{C}^n$. For each $\ell=0, \cdots, n$, set $\gamma_l:=  (i\ddbar \varphi)^{\ell} \wedge \beta^{n-\ell} $. Then there exists  a constant $B>1$ (depending on $c, C>0$) such that in $\Omega \setminus \{\psi=-\infty\}$ one has
    \begin{equation}
    	\label{eq: double estimate}
 \frac{1}{B} \sum_{\ell=0}^n \gamma_{\ell}=\frac{1}{B} \sum_{\ell=0}^n (i\ddbar \varphi)^{\ell} \wedge \beta^{n-\ell}  \leq (A\omega + i \ddbar \psi )^n \leq  B \sum_{p=0}^n (i\ddbar \varphi)^{\ell} \wedge \beta^{n-\ell}=B \sum_{\ell=0}^n \gamma_{\ell}.
    \end{equation}

    By definition of $\alpha_{j,k}$ it follows that the $(\ell,0)$-forms  $\alpha_{j_1,k_1}\wedge ...\wedge \alpha_{j_{\ell},k_{\ell}}$ are of the type $\sum  F_k  dz_{I_k}$, where $|I_k|=\ell$, and each $F_k$ is holomorphic in $\Omega$. 
    By the above identity in \eqref{computation analytic}, each $\gamma_{\ell}$ is the sum of $(n,n)$-forms of type $|F|^2 h^{-2\ell}\beta^n$, where $F$ is holomorphic in $\Omega$. By the first estimate in \eqref{eq: double estimate} it follows that for each $\ell$,
  $$ \int_\Omega |F|^2 h^{-2\ell}\beta^n \leq B \int_\Omega \rho_A \omega^n< +\infty,$$ hence  $|F|^2 e^{-2 \ell \log h}$ is integrable in $\Omega$. From the resolution of Demailly's strong openness conjecture \cite{Dem}  due to Guan-Zhou \cite{GZh} (see also \cite{Pham} for an alternative proof) it follows that each $|F|^2h^{-2\ell}$ is in $L^{p}(U, dV)$ for some $p>1$  and a smaller neighborhood $U\subset \Omega$ of $x$. 
    Finally, from the second estimate in \eqref{eq: double estimate}  we see that $\omega_{\psi}^n/\omega^n \in L^p(U, dV)$, that is what we wanted.     
  \end{proof}

\section{Log-concavity of non-pluripolar products}\label{sec 5}

\begin{theorem}\label{thm: log concavity s.u.l.}
	Let $T_1,...,T_n$ be positive $(1,1)$-currents  on a compact K\"ahler manifold $X$.  Assume that each $T_j$ has potential with small unbounded locus. Then 
	\[
	\int_X \langle T_1 \wedge ...\wedge T_n\rangle \geq \bigg(\int_X \langle T_1^n\rangle \bigg)^{\frac{1}{n}} ... \bigg(\int_X \langle T_n^n\rangle \bigg)^{\frac{1}{n}}.
	\] 
\end{theorem}

\begin{proof}We can assume that the classes of $T_j$ are big and their masses are non-zero. Other\-wise the right-hand side of the inequality to be proved is zero.  Consider smooth closed real $(1,1)$-forms $\theta^j$,  and $u_j \in \psh(X,\theta^j)$ with small unbounded locus such that $T_j = \theta^j_{u_j}$.  

For each $j=1,...,n$  Theorem \ref{thm: existence lambda=0} insures that there exists a normalizing constant $c_j>0$ and $\varphi_j\in \mathcal{E}(X,\theta^j,P_{\theta}[u_j])$ such that $\big(\theta^j_{\varphi_j}\big)^n = c_j\omega^n$. 

We can assume that $\int_X \omega^n=1$, thus we can write 
$$c_j=\int_X \big(\theta^j_{\varphi_j}\big)^n=   \int_X \big(\theta^j_{P_{\theta}[u_j]}\big)^n=\int_X \big(\theta^j_{u_j}\big)^n= \int_X \langle T_j^n \rangle .$$ 
A combination of Proposition \ref{prop: comparison generalization} and Theorem \ref{thm: lsc of MA measures} then gives
$$
\int_X \theta^1_{\varphi_1} \wedge ... \wedge \theta^n_{\varphi_n}  = \int_X  \theta^1_{P_{\theta}[u_1]} \wedge ... \wedge \theta^n_{P_{\theta}[u_n]}= \int_X  \theta^1_{u_1} \wedge ... \wedge \theta^n_{u_n} = \int_X\langle T_1 \wedge .... \wedge T_n \rangle .
$$
An application of  \cite[Proposition 1.11]{BEGZ10} gives that $\theta^1_{\varphi_1} \wedge \ldots \wedge \theta^n_{\varphi_n} \geq c_1^{1/n} \ldots c_n^{1/n} \omega^n$. The result follows after we integrate this estimate. 
\end{proof}

\paragraph{Acknowledgements.} The  first named author has been partially supported by BSF grant 2012236 and NSF grant DMS--1610202. The second named author has been supported by a Marie Sklodowska Curie individual fellowship 660940--KRF--CY (MSCA--IF). 

We would like to thank Robert Berman, Mattias Jonsson and L\'aszl\'o Lempert for their insightful comments that improved the presentation of the paper. We warmly thank the referees for many suggestions improving the presentation of the paper. 
 
\let\OLDthebibliography\thebibliography 
\renewcommand\thebibliography[1]{
  \OLDthebibliography{#1}
  \setlength{\parskip}{1pt}
  \setlength{\itemsep}{1pt}
}

\noindent{\sc University of Maryland}\\
{\tt tdarvas@math.umd.edu}\vspace{0.1in}\\
\noindent{\sc IHES, Universit\'e Paris-Saclay}\\
{\tt dinezza@ihes.fr}\vspace{0.1in}\\
\noindent {\sc Laboratoire de Math\'ematiques d'Orsay, Univ. Paris-Sud, CNRS, Universit\'e Paris-Saclay,  91405 Orsay, France}\\
{\tt hoang-chinh.lu@u-psud.fr}
\end{document}